\documentclass{amsart}
\usepackage{amscd,amssymb,amsmath,amsthm}
\usepackage[dvips]{graphicx}
\usepackage[all]{xy}
\theoremstyle{plain}
\newtheorem{definition}{Definition}
\newtheorem{proposition}{Proposition}
\newtheorem{theorem}[proposition]{Theorem}
\newtheorem{corollary}[proposition]{Corollary}

\newtheorem{lemma}[proposition]{Lemma}
\newtheorem{remark}[proposition]{Remark}
\newtheorem*{proposition*}{Proposition}
\newtheorem*{theorem*}{Theorem}
\newtheorem*{corollary*}{Corollary}
\newtheorem*{lemma*}{Lemma}
\newtheorem*{remark*}{Remark}
\newtheorem*{example*}{Example}

\newcommand{\Z}{\mathbb{Z}}
\newcommand{\Q}{\mathbb{Q}}
\newcommand{\R}{\mathbb{R}}

\begin{document}

\title{Open Gromov-Witten theory without Obstruction}

\author{Vito Iacovino}


\email{vito.iacovino@gmail.com}

\date{version: \today}


\begin{abstract}

Open Gromov-Witten invariants are defined as cycles of the  multi-curve chain complex, well defined up to isotopy. 

\end{abstract}

\maketitle

\section{Introduction}

For $X$  Calabi-Yau threefold, closed Gromov-Witten invariants count $J$-(psudo-)holomorphic curves without boundary in $X$, and are independent of the choice of almost complex structure $J$. In a similar way, if $L$ is an oriented Lagrangian submanifold of $M$ of Maslov index zero, we would like to define open Gromov-Witten invariants which count $J$-(pseudo-)holomorphic curves on $M$ with boundary on $L$. This is a much more difficult problem than the closed case since the moduli spaces have no boundary in the closed case, but have boundary and corners in the open case. The presence of the boundary makes the usual virtual counting meaningless since it depends on the choice of perturbation of the moduli space.


Motivated by the result of \cite{W}, in \cite{OGW1} we introduced the moduli space of Multi-(pseudo-holomorphic-)Curves that allow to treat the problem of boundaries at homological level, making the problem someway more similar to the closed case.  The moduli space of multi-curves are associated to certain   decorated graphs and are built from the moduli space of usual pseudo-holomorphic curves. The key point is that the perturbation of the moduli space of multi-curves is constrained in a suitable way. As an application,  in \cite{OGW1} in the case that $L$ has the rational homology of a sphere,  Open Gromov-Witten invariants to all genus are defined in terms of linking numbers of the boundary components of the multi-curve, 

In this paper we introduce the Multi-Curve chain-complex defined in terms of elementary relations which are directly related to the relations defining the perturbation of the moduli space of multi-curves. We show that Open Gromov-Witten invariants are naturally defined as $MC$-Cycles. To construct a $MC$-cycle from the moduli space of (multi-)curves it is necessary to  make some choices (such as the complex-structure and the perturbation of the moduli of multi-curves). Different choices leads to isotopic cycles, defined up to isotopy of isotopies.  




As application of our main result we define rational numbers invariants for each defined Euler Characteristic. The existence of these invariants does not rely on the existence on a bounding chain,  there is not obstruction in the definition of these invariants.
This fact is quite remarkable since the vanishing of the  obstruction classes arising in the definition of a bounding chain (familiar from Lagrangian Floer Homology \cite{FO3}) is a very strong constrain and it can be rephrased as the vanishing of the disk invariants for every  $\beta \in H_2(M,L)$ with $\partial \beta \neq 0$ in $H_1(L,\Q)$ (see \cite{obs}).

\subsection{Statement of the results}

In this paper we show that open Gromov-Witten theory leads to a cycle on the multi-curve chain complex. The Multi-Curve Cycle is well defined up to isotopy.  From a $MC$-cycle we can construct a simpler object which we call nice $MC$-cycle.  In this introduction we provide the definition of nice $MC$-cycles and their isotopies.


\subsubsection{Nice Multi-Curve Cycles}


Fix an oriented three manifold $M$ and a class $\gamma \in H_1(M, \Z)$. 
Consider the  objects  
\begin{equation} \label{generator-nice}
( H , ( w_h  )_{h \in H} ) 
\end{equation}
where $H$ is a finite set , $\{ w_h  \}_{h \in H} $ are closed one  dimensional  chains on $M$ on the homology class $\gamma$  which are close on the $C^0$-topology. 
Denote by $\mathfrak{Gen}(\gamma)^{\dagger}$ the set of the objects (\ref{generator-nice}) modulo the obvious equivalence relation:  $( H , ( w_h  )_{h \in H} ) \cong ( H' , ( w_h'  )_{h \in H'} )$ if there exists an identification of sets $H =H'$ such that $w_h=w_h'$ for each $h \in H$. 

Define $\mathfrak{Gen}(\gamma) \subset \mathfrak{Gen}(\gamma)^{\dagger}$ as the subset obtained imposing on the elements (\ref{generator-nice}) the extra condition  
\begin{equation} \label{generator-transversality}
w_h \cap w_{h'} = \emptyset \text{    if    } h \neq h'.
\end{equation}
The vector space $\mathcal{Z}_{\gamma}$ of $\mathbf{nice \thickspace MC-cycles}$ in the homology class $\gamma$  is the formal vector space generated by $\mathfrak{Gen}(\gamma)$.

Isotopies of nice $MC$-cycles are defined as follows.
Consider one parameter family of objects (\ref{generator-nice}) 
\begin{equation} \label{generator-nice-iso}
 ( H ,(\tilde{w}_h  )_{h \in H} , [a,b] ) 
 \end{equation}
where, for each $h$, $\tilde{w}_h = (  w_h^t )_{t \in [a,b]}$ is piecewise smooth one parameter family of closed one  dimensional  chains on $M$ on the homology class $\gamma$, parametrized by the interval $[a,b]$. We refer the coordinate parameterizing the interval  $[a,b]$ as time. 
Denote by $\tilde{\mathfrak{Gen}}(\gamma, [a,b])^{\dagger}$ the set of objects (\ref{generator-nice-iso}) modulo the obvious equivalence relation. 

Let $\tilde{\mathfrak{Gen}}(\gamma, [a,b]) \subset \tilde{\mathfrak{Gen}}(\gamma, [a,b])^{\dagger}$ be the subset obtaining imposing on the elements (\ref{generator-nice-iso}) the extra condition 
\[
\tilde{w}_h \cap\kern-0.7em|\kern0.7em \tilde{w}_{h'} \text{   if   }h \neq h',
\]
when considered as chains on $M \times [a,b]$.
Set $ \tilde{\mathfrak{Gen}}(\gamma) = \sqcup_{a < b} \tilde{\mathfrak{Gen}}(\gamma, [a,b])$.

Note that an element of  $\tilde{\mathfrak{Gen}}(\gamma, [a,b])$ and an element of $\tilde{\mathfrak{Gen}}(\gamma, [b,c])$, if they agree on $b$, can be glued to obtain an element of $\tilde{\mathfrak{Gen}}(\gamma, [a,c])$ . 

An $\mathbf{isotopy \thickspace of \thickspace nice \thickspace MC-cycles }$  $\tilde{Z}$ (in the homology class $\gamma$) is defined by  a formal linear combination of obejcts $\tilde{\mathfrak{Gen}}(\gamma)$:
$$\tilde{Z}= \sum_{i \in I} r_i ( H_i , ( \tilde{w}_{h,i}  )_{h \in H_i} , [a_i,b_i] ) $$
for some $r_i \in \Q$.
We consider the linear combinations modulo gluing. We assume the following constrain.

Observe that  $\tilde{Z}$ defines a one parameter family  of elements of $Z_t \in \mathcal{Z}_{\gamma}$, $t \in [0,1]$, which can be discontinues on a finite number of times (we do not define $Z_t$ if $t$ is a discontinuity point). 
If there exist $h,h' \in H_i$ such that $ \tilde{w}_{h_1,i},  \tilde{w}_{h_2,i}$ cross transversely at a time $t_0$, we require that $\tilde{Z}$ jumps according to the formula
\begin{equation} \label{jump}
 Z^{t_0^+} - Z^{t_0^-} =  \pm r_i   (H_i , ( w_{h,i}^{t_0}  )_{h \in H_i \setminus \{h_1,h_2 \} }  )  .
\end{equation}
where the sign is defined by the sign of the crossing.
Denote by  $\tilde{\mathcal{Z}}_{\gamma}$ the set of isotopies of nice $MC$-cycles.

If we consider $\mathfrak{Gen}(\gamma) $  as elements of the  symmetric tensor product of one chains $\text{Sym} (C_1(M))$ using 
$( H , \{ w_h  \}_{h \in H} )  \mapsto \otimes_h w_h $,  
formula  (\ref{jump}) can be written as 
\begin{equation} \label{equivalence1}
   w_{h_1}^{t_0^-} \otimes w_{h_2}^{t_0^-}   \leadsto    \tilde{w}_{h_1}^{t_0^+} \otimes \tilde{w}_{h_2}^{t_0^+}  \pm 1  .
\end{equation}







\subsubsection{Cycles with not zero Chern Class}

We need to extended the slightly the definition of $MC$-cycle provided above. It will depend on a homology class $\mathfrak{c} \in H_1(M, \Z)$ which we call the Chern-Class of the $MC$-cycle. The $MC$-cycle defined above have  $\mathfrak{c} =0$. 

Fix  $\mathfrak{c} \in H_1(M, \Z)$ . Let $\mathfrak{Gen}(\gamma| \mathfrak{c})^{\dagger}$ be the set of objects 
\begin{equation} \label{generator-nice-ann}
 ( H , ( w_h^{\flat})_{h \in H} , (w^{ann}_h )_{h \in H}  ) 
 \end{equation}
 modulo the obvious equivalence relation, 
where 
\begin{itemize}
\item $H$ is a finite set ,
\item  for each $h \in H$, $ w_h  \in C_1(M,  \Z )$ is a closed one  dimensional  integer chain   on $M$ with $[w_h] = \gamma$,  
\item  $ w_h^{ann} \in C_1(M,  \frac{1}{2} \Z )  $ is a closed one  dimensional  half-integer chain  on $M$ with $2 [w_h^{ann}] = \mathfrak{c}$, 
\item  the one chains $\{ w_h  \}_{h \in H} $ are close on the $C^0$-topology, 
 \item the one chains $\{ w_h^{ann}  \}_{h \in H} $ are close on the $C^0$-topology. 
\end{itemize} 
As above, we consider (\ref{generator-nice-ann}) modulo the obvious equivalence relation.

Set $w_h = w_h^{\flat} + w^{ann}_h$.
 Define $\mathfrak{Gen}(\gamma| \mathfrak{c}) \subset \mathfrak{Gen}(\gamma| \mathfrak{c})^{\dagger}$  imposing  again condition (\ref{generator-transversality}) on the elements (\ref{generator-nice-ann}) . Define $\mathcal{Z}_{\gamma| \mathfrak{c}}$ as the formal $\Q$-vector space generated by $\mathfrak{Gen}(\gamma|\mathfrak{c})$.
We refer to $\mathfrak{c}$ as the Chern class of the $MC$-cycle.

The vector space of isotopies $\tilde{\mathcal{Z}}_{\gamma | \mathfrak{c} }$ can be defined making the same modifications to the definition of $\tilde{\mathcal{Z}}_{\gamma }$ .

\begin{remark}
The class $\mathfrak{c}$ is associated to the perturbed moduli space of area zero annulus with one boundary marked point. More precisely we need to consider the moduli space of multi-curves associated to the following decorated graphs with area zero:
\begin{enumerate}
\item annulus with one boundary marked point;
\item  disk with one internal and one boundary marked point;
\item  disk with three boundary marked points and one edge connecting two marked points.
\end{enumerate}
 If we consider the last graph with fixed cyclic order of its boundary marked points we obtain an Euler structure on $L$ .  $\mathfrak{c}$ is the Chern Class of this Euler Structure and it is obtained considering the last graph without cyclic order.  $w^{ann}$ is an half-integer chain since the moduli associate this graph acquire a $\Z_2$ automorphism group when we forget the cyclic order of its boundary marked points . 
\end{remark}





\subsubsection{Main Theorem}

Let  $X$ be a Calabi-Yau simplectic six-manifold and a let $L$ be a Maslov index zero lagrangian submanifold of $X$. As in \cite{OGW1}, to deal with the nodes of type $E$ in sense of \cite{L}, we assume $[L] =0 \in H_3(X, \Z)$. 
Fix  a four chain $K$ with $\partial K =L$.

To the four chain $K$, using (\ref{chern-four-chain}), we can associate a one dimensional homology class $\mathfrak{c}(K) \in H_1(L,\Z)$, which we call Chern class. 





\begin{theorem} \label{main-theorem}
Let $\beta \in H_2(X,L, \Z)$ and $\chi \in \Z$ . To the moduli space of pseudoholomotphic multi-curves of homology class $\beta$ and Euler characteristic $\chi $ it is associated a multi-curve cycle $Z_{\beta, \chi} \in \mathcal{Z}_{\partial \beta| \mathfrak{c}}$ with Chern class $\mathfrak{c}$. $Z_{\beta, \chi}$ depends by the varies choices we made to define the  Kuranishi structure and its perturbation on the moduli space of multi-curves.
Different choices lead to isotopic $MC$-cycles.
\end{theorem}
The dependence on the various choices appearing in Theorem \ref{main-theorem} is analogous to one appearing in the context  of $A_{\infty}$-structures  associated to pseudo-holomorphic disks (\cite{FO3}) .


For technical reasons it is problematic to construct   a nice-$MC-cycle$ in the sense above directly from the moduli space of multi-curves (this problem is related to the issues of remark \ref{odd} below). 
In section $2$ we introduce a  more general notion of Multi-Curve-Cycle. 
In section $3$ we introduce the moduli space of pseudoholomorphic Multi-Curves. The perturbations of the moduli of multi-curves are constrains in a way  directly related to the relations defining $MC$-cycles
leading to a relatively easy construction of the $MC$-cycle associated to the pair $(X,L)$. 
 Finally, an elementary topological argument shows that from   $MC$-cycle we can construct a nice $MC$-cycle in the sense above. The same considerations work for isotopies .

\begin{remark} \label{odd}
For any boundary marked point, there are two interesting properties which we can require to the  Kuranishi structure of moduli space of  pseudo-holomorphic curve:   weakly sub-mersivity and forgetful compatibility. 
These two conditions are some-way in odds, in the sense that it is not possible to impose both simultaneously.  

In the constructions of an $A_{\infty} $ structure of Fukaya-Oh-Ohta-Ono  it is necessary to work \emph{component-wise} on of the moduli  space of disks.  
This sets strong constrains on the perturbation of the moduli space and it does not  allow to distinguish among marked and singular points.  
This issue is also related to the difficulty to keep symmetries of the problem (such as cyclic symmetry) in the standard approach to the construction of $A_{\infty} $ structures .

In contrast, in our  construction of the $MC$-cycle using the moduli space of multi-curve we do not need to work component-wise, allowing to deal with these issues in a more simple and direct way. This simplify drastically the construction of a $MC$-cycle compared to the construction of an $A_{\infty} $ structure of Fukaya-Oh-Ohta-Ono.
  
\end{remark}

\subsubsection{Rational invariants}

In order to define rational numbers we use the elementary fact that if $\gamma=0 \in H_1(L,\Q)$ there exists an homomorphism 
\begin{equation} \label{MultiLink}
MultiLink: MCH_0(\gamma) \rightarrow H_0(point, \Q)
\end{equation}
defined from the usual linking number of curves on $L$. As corollary of Theorem \ref{main-theorem} we obtain
\begin{corollary} \label{rational}
To each $\beta \in H_2(X,L, \Z)$ with $2\partial \beta + \mathfrak{c}=0 \in H_1(L,\Q)$  it is associated a rational number $a_{\beta, \chi} \in \Q$ which depends only on $(X, L )$ and $K$, and is independent of compatible almost complex structure $J$ and various choices we made to define a Kuranishi structure. 
\end{corollary}


\subsection{Relation with String Theory}  \label{string}
 At physical level of rigour, it is well known that the partition function of the closed topological string defined  using the path integral can be computed in terms of closed Gromov-Witten invariants. Holomorphic curves arise at physical level from the supersymmetric localization of the mathematical ill defined path integral. 
In analogy with the closed case it should be expected that Open Gromov-Witten invariants are the mathematical counterpart of the Open Topological String partition function.

In \cite{W} Witten defines at the physical level of rigor the open topological string partition function for a pair $(X,L)$, where $X$ is a Calabi-Yau three-fold and $L$ is Maslov index zero lagrangian submanifolds of $X$. 
Witten shows that in the open case the supersymmetric localization is more subtle compared to the  closed case. He argues that the open topological string partition function receives contributions also from degenerate or partially degenerate curves. 
In the case of the cotangent bundle $M=T^*L$ (where non-constant holomorphic curves are absent) the open topological string is equivalent to Chern-Simons theory on $L$. The degenerate curves correspond to Feynman graphs of the Chern-Simons theory. 

On a general Calabi-Yau three-fold  the Chern-Simons theory of $L$ needs to be corrected by Wilson loops associated to the boundary circles of the curves. The contribution of a set of holomorphic curves $(\Sigma_i,u_i$ can be computed  as expansion on the formal parameter $\theta$ using formula $(4.50)$ of \cite{W}:
\begin{equation} \label{witten}
\mathcal{L} = \frac{1}{2 g_s} \int_L Tr (A \wedge dA + \frac{2}{3} A \wedge A \wedge A ) 
+ \sum_i g_s^{-\chi_i}   \eta_i 
\text{exp} (- \theta q_i)
Tr P \exp \int_{w_i} A,
\end{equation}
where $w_i= \partial \Sigma_i$, $q_i$ is the symplectic area of $\Sigma_i$, $\eta_i$ is the sign factor, $\chi_i$ the euler characteristic of $\Sigma_i$. Here we have made explicit the dependence on the parameter $g_s$ (the string coupling) since we are interested to the formal expansion on it. $P \exp \int_{w_i} A$ is the path-order integral, in the notation used by physicists for the homolonomy of $A$.

The starting point of the papers \cite{OGW1} , \cite{OGW2} was to provide a mathematical definition of (\ref{witten}). The Feymann expansion of formula (\ref{witten}) leads to the definition of     moduli space of multi-curves. $MC$-cycle should be considered as the mathematical formulation of point splitting pertubative Chern-Simons. 
These results should be considered as the natural mathematically interpretation of the claim of \cite{W}, and the starting point of Open Gromov-Witten Theory.

The mathematically most interesting case is the case of gauge group $U(1)$. Actually the non-abelian case $U(N)$ follows from the abelian case.  Namely the $U(N)$ expansion associated to $L$ can be reproduced from the $U(1)$ expansion associated to $\tilde{L}$, where $\tilde{L}$ is a brunched covering of $L$ in a tubular neighborhood of $L$ identified with $T^*L$.

\subsubsection{Numerical Invariants}

To define a numerical invariant it is necessary one more step, which was only implicit in \cite{W}. The $BV$ quantizaton of Chern-Simons theory leads to a measure on the moduli space of flat connections which is defined up to an exact form (see \cite{CS}). The partition function is obtained from the integration of this measure on the moduli space of flat connections. Let us focus on the abelian case.

Instead to integrate on the whole moduli space of flat connections,
we  do a little better, and we integrate over a connected component of it.  
The moduli space of abelian flat connections can be identified with the set of $\text{Hom}(H_1(L, \Z), U(1))$, and its connected components are given by the fibers of the map
\begin{equation} \label{flat}
   \text{Hom}(H_1(L, \Z), U(1)) \rightarrow  \text{Hom}(Tors(H_1(L,\Z)), U(1))  .
\end{equation}

In the Feymann expansion of (\ref{witten}) around a flat connection $A_0$ arise the classical factor  $\exp(i \int_{w} A_0)$, where $w = \sum_i w_i$ is the total boundary of the multi-curve. In the abelian case the propagator does not depend on $A_0 $ and  only the classical numerical factor dependence on $A_0$  . It follows that the integral over a fiber of (\ref{flat}) can be factorized as:
\begin{equation} \label{Z}
\Bigl( \int_{\text{connected component} }    \exp(i \int_{\gamma} A_0) \Bigr)  \times    (\text{Feymann diagrams}).
\end{equation}
 It is elementary to check that
$$  \int_{\text{connected component} }   \exp(i \int_{\gamma} A_0)  \neq 0  \Leftrightarrow  [w] \in Tors(H_1(L,\Z))  .$$





Thus we expect that the contribution of the Feymann diagrams on the second factor of (\ref{Z}) leads to an invariant if $\partial \beta \in Tors(H_1(L,\Z))$. This fact is proved mathematically in this paper


The shift by the Chern Class appearing in Corollary \ref{rational} is explained at the physical level using the interpretation of subsection \ref{B-field}. Namely, we need  we need to expand around a connection $A_0$ which curvature $F_0$ satisfying the relation
$F_0 + B|_L  =0 .$
From this we conclude that $ \partial \beta + \frac{1}{g_s} PD_L(B|_L) =0 $ in $ H_1(L,\Q)$.


\subsubsection{ Four Chains} \label{B-field}

As in \cite{OGW1}, to deal with the  boundary nodes of types $E$ (\cite{L})  we assume that $L$ is homologically trivial in $H_3(M, \Q)$ and we fix a four-chain $K$ in $X$ such that $\partial K = L$ holds.
In this subsection we explain the physical origin of the four chain. Roughly, the four chain should be seen as a deformation of the geometry determined by a bulk deformation with a prescribed singularity  along $L$.

As already observed in \cite{OGW1} the existence of $K$ is related to the fact that topological charge of the brane must be zero in  the compact case. The topological charge is provided by   background fields of string theory which can be represented as differential forms. In presence of a brane these differential forms  develop a singularity along the brane. The branes of the $A$-model topological string are associated to lagrangian submanifolds of Maslov index zero. 
In presence of $A$-brane the complexified Kahler form $\omega + i B$ (where we denote by $B$ the B-field) acquire quantum corrections in the string coupling $g_s$ and  a flux such that 
\begin{equation} \label{B-constrain}
 \int_{S^2} b =g_s
\end{equation}
where $S^2$ is any small sphere surrounding $L$. 

\begin{remark}
    Using the  duality between $A$-model and $M$-theory, the singularity of (\ref{B-constrain}) is equivalent to the singularity  of the closed differential four-form $G$  of $M$-theory in presence of a $M5$-brane.  We stress that this duality is a conjecture also at physical level of rigour. 
    
     The singularity of (\ref{B-constrain})  is also dual to the singularity of the holomorphic three form $\Omega$  of the $B$-model in presence of a $B$-brane.
\end{remark}

 Considering $b = \sum_{i \geq 0} g_s^{i} b_i $ as a formal power series in $g_s$ with coefficients differential forms $b_i$ of degree two ,     $b_1$ has a singularity constrained by (\ref{B-constrain}), while
 $b_i$ is smooth for $i \neq 1$.




In order to avoid to deal with singularities we consider the differential geometric blow-up $\hat{X}$ of $X$ along $L$. $\hat{X}$ is a six-dimensional manifold with boundary,  which internal part is $X \setminus L$ and which boundary is the spherical normal bundle $S(N(L \subset X))$ of $L$. $\hat{X}$ is also diffeomomorphic to  the complementary in $X$ of a small tubular neighborhood of $L$. 

The advantage to work with $\hat{X}$ is that on $\hat{X}$ there exist smooth differential  two-forms $b$ satisfying condition (\ref{B-constrain}). Denote by $D^3L$ the disk bundle of $TL$, and let $S^2L = \partial D^3L$ be its boundary. Identify $D^3L$  with a small tubular neighborhood of $L$ in $ X$. 
$b_1$ defines an element of $H^2(X')$ and condition (\ref{B-constrain}) is rewritten as  
\begin{equation} \label{b-singularity}
 b_1 \mapsto 1 \text{  throught   }H^2(\hat{X}) \rightarrow H^2(\partial \hat{X}) \rightarrow  H^2(S^2L) / H^2(D^3L) \cong \Z  .
\end{equation}

\begin{remark}
The Mayer-Vietors sequence for the pair $(X' , T)$ follows the exact sequence: 
$$0 \rightarrow H^2(X) \rightarrow H^2( X') \rightarrow  H^2(S^2L) / H^2(D^3L) \rightarrow 0.$$
Here we use the surjectevity of $H^1(T) \rightarrow H^1(S)$ and the fact that $H^2(S^2(L)) \rightarrow H^3(X)$ has image to $PD([L])$, that vanishes since our assumption $[L]=0$.
\end{remark}

The homology class of the four chain $K$ is thus defined as the Poincare Dual of $[b_1] $
$$ [K] = PD([b_1]) \in H_4(\hat{X} , \partial \hat{X}) = H_4(X,L),$$
and  $ \partial [K] = [L] \in H_3(L) $ is equivalent to relation (\ref{b-singularity}).

It is standard to weight the contribution of a curve $(\Sigma,u)$  by the factor $ \text{exp} \left( i \int_{\Sigma} u^*(\omega) \right) $.  Since in our case the simplectic form $\omega $ is singular 
this factor does not have an immediate interpretation . Nevertheless it is possible to give a meaning to the full expansion  in a more refined way using the fact that the jump of this factor appearing when an internal point of $\Sigma$ crosses $L$ is  compensated by the creation of  a pseudo-holomorhic curve with a boundary node of type $E$. 
Morally we could write $\int_{\Sigma} u^*(b_1) = K \cap u(\Sigma)$, which makes  clear that it is not well defined since both $\partial K$ and $\partial \Sigma$ are not empty, and actually $\partial u( \Sigma) \subset \partial K $.  
To proceed mathematically, 
in this paper (as in \cite{OGW1}) we consider moduli space of  multi-curves with internal marked points   mapped on $K$. With the formalism of the multi-curves   the  problem is solved systematically  as consequence of the fact that we obtain a  moduli space which is closed in a suitable sense .  Roughly the ambiguity to define $ K \cap u(\Sigma)$ is cancelled by the jump problem mentioned before for the (perturbed) moduli space of constant maps. 



\begin{remark}
The problem mentioned above simplifies drastically in the case of genus zero Open Gromov-Witten invariants. In this case, we  consider  the moduli of holomorphic disks and the moduli of holomorphic spheres with one internal marked point mapped on $K$. Since a sphere does not have boundary the issue mentioned above does not appear and a sphere can cross $L$ only in a codimension one region of the parameters (such as the complex structure of $X$). This allow to treat the moduli space of  disks independently and the issue mentioned above reduces simply to a wall crossing phenomena. In contrast, in higher genus the problem appears in the whole space of parameters, the area zero curves play a more important rule, and we are forced to use the four chain $K$ from the beginning in a more fundamental way to define the correct moduli space.
\end{remark}




\begin{remark} 
Property $(\ref{B-constrain})$ plays an important role also in the context of large $N$-duality (\cite{OV}), where it can be formulate as the invariance of the flux of the $B$-field at infinity before and after the conifold transition. This match is important when we take the geometric transition of \cite{OV} as a local model for geometric transitions in more general manifolds.
\end{remark}

\section{Multi-Curve Chain Complex}

In  this section we define Multi-Curve Homology, that is the main algebraic topological object we shall use to define Open Gromov-Witten invariants. That is associated to the following data: 
\begin{itemize}
\item An oriented three manifold $M$,
\item a finite-rank free abelian group $\Gamma$, called \emph{topological charges},
\item an homorphism of abelian groups
$$ \partial : \Gamma \rightarrow H_1(M, \Z) $$
called boundary homomorphism.
\item an homomorphism of abelian groups 
$$ \omega : \Gamma \rightarrow \R $$
called \emph{symplectic area}.
\end{itemize}


\subsection{Decorated Graphs}
A decorated graph $G$ is defined  by  an array
$$(V(G), \{ H_v \}_{v \in V(G)}, E(G), ( \beta_v )_{v \in V(G)}, ( \chi_v )_{v \in V(G)}) $$ 
where 
\begin{itemize}
\item $V(G)$ a finite set, called vertices;
\item for each $v \in V(G)$, $H_v$ is finite set, called half-edges ; 

Set $H(G)= \sqcup_{v \in V(G)} H_v$.
\item  $E(G)$ is partition of $H(G)$ in sets of cardinality one or two, called edges.  The sets of cardinality two are called internal edges $E^{in}(G)$, the sets of cardinality one are called external edges $E^{ex}(G)$; 
\item for each $v \in V(G)$, $\beta_v \in \Gamma$  is called topological charge and $\chi_v \in \Z_{\leq 1}$ is called Euler characteristic of the vertex $v$.
\end{itemize}


We assume that
$$ \beta_v \in \Gamma_{tors} \Rightarrow \beta_v=0 $$

The homology class of $G$ is defined by 
$$\beta = \sum_{v \in V(G)} \beta_v,$$
and its Euler characteristic by 
$$\chi(G)= \sum_{v \in V(G)} \chi_v - |E(G)|  .$$

A vertex $v$ is called unstable if $\beta_v = 0$ and $2 \chi_v - | H_v | \geq 0 $.
The graph $G$ is called stable if all its vertices are stable.




Fix a norm $ \parallel \bullet \parallel$ on $\Gamma_{\R}= \Gamma \otimes  \R$.  For each positive real number $C^{supp} \in \R_{>0}$  denote by $ \mathcal{G}(\beta, \kappa,C^{supp})$ the set of stable decorated graphs with topological charge $\beta$ with $|E^{ex}(G)| - \chi(G)= \kappa$ and
 \begin{equation} \label{support}
\parallel \beta_v \parallel \leq C^{supp}  \omega(\beta_v)
\end{equation}
 for each  $ v \in V(G)$.


Observe that $\mathcal{G}(\beta, \kappa, C^{supp}) $ is a \emph{finite} set.
In the following of this section we fix the constant $C^{supp}$ and we omit the dependence on  $ \parallel \bullet \parallel$ and $C^{supp} $ in the notation.

\subsubsection{operation $\delta_e$}

Let $ e=\{ h_1,h_2 \} \in E_{in}(G) $ be an internal edge of $G$. Define a graph $\delta_e G$ in the following way. We have different cases:


Let $h_1 \in H_{v_1}$ and $h_2 \in H_{v_2}$ for some  $v_1 , v_2 \in V(G)$ with $v_1 \neq v_2$.  $V(\delta_e G)$ is given by replacing $v_1$ and $v_2$ in $V(G)$  by a unique vertex $v_0$.   Define
$$ H_{v_0} =   H_{v_1}  \sqcup  H_{v_2} \setminus  \{ h_1,h_2 \} .$$
$$ \chi_{v_0} =   \chi_{v_0}+ \chi_{v_1} -1, \beta_{v_0} = \beta_{v_0} + \beta_{v_1} .$$

If $h_1 , h_2 \in H_{v_0}$ for some $v_0 \in V(G)$,
$V(\delta_e G) = V(G)$, $\chi_{v_0}$ is decreased by one, and all the  other data remain the same.


\subsubsection{Graphs with marked edges}


Let  $ \mathcal{G}_{l}(\beta, \kappa)$ be the set of pairs $(G, m)$, where $G \in \mathcal{G}(\beta, \kappa)$, and $m=  \{ E_0, E_1,..., E_l \})$ is an increasing sequence of subsets $E_0 \subset E_1 \subset ... \subset E_l$ of $E(G)$.  We consider  $ \mathcal{G}(\beta, \kappa)$ as a subset of $\mathcal{G}_{0}(\beta, \kappa)$ 
identifying $G \in \mathcal{G}(\beta, \kappa)$ with $ (G, \{ \emptyset \}) \in \mathcal{G}_{0}(\beta, \kappa)$,
 

On the set of graphs  $ \mathcal{G}_*(\beta, \kappa)$ we extend the operator $\delta$ defined above as follows.


For $e \in E(G) \setminus E_l$, define $\delta_e(G, m)= (\delta_e G, m) \in \mathcal{G}_{l}(\beta)$, where to define $m$ on the right side we use the identification  $E(G) = E(\delta_e G)  \sqcup \{ e \}$.


In $ \mathcal{G}_*(\beta, \kappa)$ we can moreover define the following operation: for $0 \leq i \leq l$ define $\partial_i (G, \{ E_0, E_1,..., E_l \}) = (G, \partial_i  \{ E_0, E_1,..., E_l \}) \in \mathcal{G}_{l-1}(\beta)$, where $\partial_i  \{ E_0, E_1,..., E_l \} =  \{ E_0,...,\hat{E_i},..., E_l \}$


\subsection{Multi Curve Chain Complex}

We now define a complex $(( \mathcal{C}_d^{MC} )_d,\hat{\partial})$.

For  a finite set $S$, we denote by $\mathfrak{o}_S$ the set of ordering of $S$ up to parity.


For a decorated graph $G$ and $e \in E^{in}(G)$, we denote by $\pi_e : M^{H(G)} \rightarrow M \times M $ the projection to the components associated to $e$ and by $Diag$ the diagonal of $M \times M$.

An object $C \in \mathcal{C}^{MC}_d$ is given by a collection 
$$C_d=  (  C_{G,m}  )_{G,m} ,$$
where for each $(G,m)$
$$C_{G,m} \in C_{|H(G)| + |m| +d}(M^{H(G)}, \mathfrak{o}_{H(G)}  )^{Aut(G,m)}.$$ 
is a $Aut(G)$-invariant chain in $M^{H(G)}$ with coefficients in $\mathfrak{o}_{H(G)}$ of dimension  $|H(G)| + |m| +d$. We require that  
\begin{enumerate}
\item  $C_{G,m}$ is transversal to the big diagonals associated to $E^{in}(G) \setminus E^l$ as stratified space. That is, $C_{G,m}$ is transversal  to   $ \sqcap_{e \in E''} \pi_e^{-1}(Diag)$ for each subset $E' \subset E^{in }(G)  \setminus E^l$;
\item forgetful compatibility holds in the sense of subsection \ref{forget-MC}.
\item $C_{G,m}= C_{cut_{E_0}G,m}$, where $cut_{E_0}G$ is the graph obtained cutting the edges $E_0$. Hence $E^{in }(cut_{E_0}G)= E^{in}(G) \setminus E_0, H(cut_{E_0}G)= H(G)$.  
\end{enumerate}


The operator $\hat{\partial}$ is defined by
\begin{equation} \label{derivation0}
\hat{\partial} = \partial + \delta + \eth : \mathcal{C}_{d}(\beta,\chi)  \rightarrow \mathcal{C}_{d-1}(\beta,\chi)
\end{equation}
where:
\begin{itemize}
\item $ \partial: \mathcal{C}_{d}(\beta,\chi) \rightarrow \mathcal{C}_{d-1}(\beta,\chi)$  is the usual boundary operator on the space of chains;
\item 
$$(\delta C)_G= \sum_{ \delta_e (G',m')= (G,m) } \delta_e C_{(G',m')} $$
where $  \delta_e C$ is the projection on $L^{H(\delta_e G)}$ of $C \cap  \pi_e^{-1} (Diag) $.  
The orientation of f $C \cap  \pi_e^{-1} (Diag) $ is defined according to the relation $ T_*C = N_{Diag}(L \times L) \oplus T_*( C \cap  \pi_e^{-1} (Diag))  $, where $N_{Diag}(L \times L) \subset T_*(L \times L)$ is the normal bundle to the diagonal.
\item 
$ (\eth  C)_{(G,m)} = (-1)^{d+1} \sum_{0 \leq i \leq l} (-1)^i C_{(G,\partial_i m)}.$
\end{itemize}

It is easy to check that 
$$\partial^2=0, \delta^2=0, \eth^2=0, \partial \delta + \delta \partial =0, \partial \eth + \eth \partial =0,  \delta \eth + \eth \delta =0 ,$$ and therefore 
$$\hat{\partial}^2=0  .$$

\subsubsection{Operator $\delta_e$} \label{operator-delta}

Let $(G,m) \in \mathfrak{G}_l$ and $e \in E(G) \setminus E_l$. 
For $C \in C_*(L^{H(G)})^{\text{Aut}((G,m))}$ we define the chain $\delta_e C$ as follows.

Let $\text{Aut}((G,m),e) < \text{Aut}(G,m)$ be the group of automorphims of $(G,m)$  fixing the edge $e$.
Consider the chain $C \cap Diag_e$  as an element of  $C_*(L^{H(G)})^{\text{Aut}((G,m),e) }$.
The orientation of $C \cap  \pi_e^{-1} (Diag) $ is defined according to the relation $ T_*C = N_{Diag}(L \times L) \oplus T_*( C \cap  \pi_e^{-1} (Diag))  $, where $N_{Diag}(L \times L) \subset T_*(L \times L)$ is the normal bundle to the diagonal.


There is an homormorphism of groups  $\text{Aut}((G,m),e) \rightarrow \text{Aut}((\delta_e G,m))$ which, together the projection $L^{H(G)} \rightarrow L^{H(\delta_e(G))}$, induces a map of global orbitfold
$$pr: L^{H(G)}/\text{Aut}((G,m),e) \rightarrow L^{H(\delta_e G)}/ \text{Aut}((\delta_e G,m)).$$

Set
 $$\delta_e C = -pr_*( C \cap Diag_e ).$$


\subsubsection{Forgetful compatibility} \label{forget-MC}




Fix $w^{ann} : F^{ann} \rightarrow M$ a one chain where $ F^{ann}$ is one dimensional compact manifold.

Let $(G,m) \in \mathcal{G}_*$ be a decorated graph.

Let $G''$ be the decorated graph obtained from $G$ removing the half-edges $H_0$. $G''$ is not necessary a stable graph. Let $G'$ be the graph obtained from $G''$ removing the unstable vertices.  
We denote $forget_{H_0}G = G'$.

Let $V^{ann0}$ be the set of vertices $v$ of $G$ with $\chi_v=0, \beta_v=0$, $|H_v|=1$, $H_v \subset H_0$.

If $V(G'') \setminus V(G') \neq V^{ann0}$,  we require $Z_{G,m}=0$.


We assume that 
to $(G',m')$ are associated the following data:
\begin{itemize}
\item A chain 
$$Z_{G',m'}= \sum_{a \in  \mathcal{A}} \rho_a  [f_a, N_a )]    \in C_*(M^{H(G')}),$$
 where $N_a$ are manifolds with corners,  $f_a: N_a \rightarrow M$, and $\rho_a \in \Q$.
\item for each $ a \in \mathcal{A}$ and $v \in  V(G')$  a one dimension compact manifold $F_{a,v}$ ,
\item   
for $ a \in \mathcal{A}, h \in H_0 ,v \in V(G')$ a smooth function  
$\overline{f}_{a,v,h} : N_a \times F_{a,v} \rightarrow M$, $\overline{f}_{a,h}^{ann}: N_a \times F_{ann} \rightarrow M$,
\end{itemize} 
such that  $ \{\overline{f}_{a,v,h} \}_h,  $ are close in the $C^0$ topology and for each $x \in N_a , v \in V(G')$ we have
$$ (\overline{f}_{a,v,h})_* [\{ x \} \times F_{a,v}] =    \partial \beta_{v}  \text{ in } H_1(L, \Z) , $$
$$ (\overline{f}_{a,h})_* [\{ x \} \times F_{ann}]     \text{    is close to $w^{ann}$  in the $C^0$-topology     } . $$



Assume that  $V(G'') \setminus V(G')= V^{ann0}$ and set 
$f_{a,h,v} = f_{a,h}^{ann} $ and $F_{a,v} = F_a^{ann} $ if $v \in V^{ann0}$ and $h \in H_v$.

We require that
$$  Z_{G,m}= \frac{1}{|H_0|!} \sum_s \sum_{a \in  \mathcal{A}} \rho_a  [f_a \times \prod_{h \in H_0} f_{a,s(h), v_h}  , N_a \times  \prod_{h \in H_0}     F_{a,v_h}  ]   .$$
where the sum is over the set of permutations $s$ of $H_0$. $v_h$ denotes the vertex to which $h$ is attached.





\begin{remark}
The definition states that  $Z_{G,m}$ can be expressed in term of a chain on $M^{H(G')} \times \mathfrak{Gen}(\beta, w^{ann})$ (here 
$\mathfrak{Gen}(\beta, w^{ann})$ is defined as in (\ref{generator2})). 
\end{remark}



\subsection{Isotopies}

An element $\tilde{C} \in \tilde{\mathcal{C}}^{MC}_d$ is given by a collection 
$$\tilde{C}_d=  (  \tilde{C}_{G,m}  )_{G,m} ,$$
where for each $(G,m)$
$$\tilde{C}_{G,m} \in C_{|H(G)| + |m|+d +1} (M^{H(G)} \times [0,1], \mathfrak{o}_{H(G)})^{\text{Aut}(G,m)}  $$
  is s a $Aut(G,m)$-invariant chain in $M^{H(G)} \times [0,1]$ with coefficients in $\mathfrak{o}_{H(G)}$ of dimension  $|H(G)| + |m| +d +1$. We require that
\begin{enumerate}
\item  is transversal to the big diagonals associated to $E^{in}(G) \setminus E^l$ as stratified space, i.e. , $\tilde{C}_{G,m}$ is transversal  to   $ \sqcap_{e \in E''} \pi_e^{-1}(Diag) \times [0,1]$ for each subset $E' \subset E^{in }(G)  \setminus E^l$;
\item satisfies the forgetful compatibility in the sense of subsection \ref{forget-MC}.
\item $ \tilde C_{G,m}= \tilde C_{cut_{E_0}G,m}$. 
\end{enumerate}

We the operators $\partial, \delta, \eth$ are extended straightforwardly to $\tilde{\mathcal{C}}$. 

Given two $MC$-cycles $Z_0$ and $Z_0$. An isotopy of $MC$-cycles between $Z_0$ and $Z_1$ is an element of  $\tilde{Z} \in \tilde{\mathcal{C}}^{MC}_0$ such that 
$$ \hat{\partial}\tilde{Z} = Z_0 \times \{ 0 \} - Z_1 \times \{ 1 \} .$$




\subsection{Nice Multi-Curve Cycles}

We now introduce a simpler version of Multi-Curve Homology and we prove that it is equivalent to the one defined in the preview section .


\begin{definition} \label{nice-cycles-definition}
A multi curve cycle $Z$ is called $\emph{nice}$ if 
$$Z_{G,m}=0  \text{   if  }   |m|>0   .$$
\end{definition}
It follows directly from the definition that $Z_{G, \{ E_0 \}}$ does not depends on $E_0$. Hence we can set $Z_G= Z_{G, \{ E_0 \}}$ for any $E^{ex}(G) \subset E_0 \subset E(G)$.





Nice Multi-Curve Cycles can be described as follows.
A generator of nice $MCH$ is defined by an array
\begin{equation} \label{generator2}
(V, H,  ( w_{v,h} )_{v \in V,h \in H} ,  ( w^{ann}_h )_{h \in H}, ( \chi_v , \beta_v )_{v \in V}  )
\end{equation}
where $V$ and $H$ are finite sets, $\beta_v \in \Gamma$,  $\chi_v \in \Z_{\leq 1}$,   $ \{ w_{v,h} \}_{h \in H}$ are closed one dimensional chain on $M$ on the homology class $\partial \beta_v$ , which are close in the $C^0$-topology.  
 $ \{ w^{ann}_h \}_{h \in H}$  are closed one dimensional chains on $M$  which are close to $w^{ann}$ in the $C^0$-topology.  

We assume that $\beta_v \in Tors (\Gamma)$ implies $ \beta_v=0 $ and $\chi_v < 0$ if $\beta_v =0$. Moreover
 $$   \parallel \beta_v \parallel \leq C^{supp}  \omega(\beta_v) \text{   for each  } v \in V(G)  .$$

Denote by $\mathfrak{Gen}(\beta, \chi)^{\dagger}$ the set of objects (\ref{generator2}) such that $\sum \beta_v = \beta, \sum_v \chi_v = \chi$, modulo the obvious equivalence relation.   

Define $\mathfrak{Gen}(\beta, \chi) $ as the subset of $ \mathfrak{Gen}(\beta, \chi)^{\dagger}$ obtained imposing the extra condition  
\begin{equation} \label{generator-transversality2}
w_{v,h} \cap w_{v',h'} = \emptyset \text{    if    } h \neq h'.
\end{equation}
The vector space $\mathcal{Z}_{\beta}^{\diamond}$ of $\mathbf{nice \thickspace MC-cycles}$ in the homology class $\beta$  is the formal vector space generated by $\mathfrak{Gen}(\beta)$.


From an  objects (\ref{generator2}) we can construct a nice $MC$-cycle in the sense of definition \ref{nice-cycles-definition} as follows. 
For a decorated graph $G \in \mathcal{G}$, 
let $ V^{ann0}$ be the set of vertices of $G$ which are area zero annulus with one boundary marked point. Assume that there are no other area zero vertices besides $ V^{ann0}$ and that there exists an identification $V(G) \setminus V^{ann0}= V$ compatible with $ ( \chi_v , \beta_v )_{v \in V} $. 
Denote
$w_{h,v} = w_h^{ann} $ if $v \in V^{ann0}$ and $h \in H_v$.
Set
\begin{equation} \label{nice-graphs}
Z_G = \sum_{s}  \prod_v \prod_{h \in H_v} w_{s(h),v}  ,
\end{equation}
where the sum is made  on the set of bijective maps $s : \sqcup_v H_v  \rightarrow H$.


\subsubsection{Isotopies}


An isotopy of nice-$MC$ cycles $\tilde{Z} $ is defined by a linear combination of objects 
\begin{equation} \label{generator2-iso}
(V, H,  ( \tilde{w}_{v,h} )_{v \in V,h \in H} ,  ( \tilde{w}_{h}^{ann} )_{h \in H}   , ( \chi_v , \beta_v )_{v \in V} , [a,b]   )
\end{equation}
where $[a,b] \subset [0,1]$,  $ \tilde{w}_{h,v} : F_v \times [a,b] \rightarrow M$ for some one dimensional compact manifold $F_v$.

Write 
$$\tilde{Z}= \sum_i r_i (V_i, H_i,  ( \tilde{w}_{v,h,i} )_{v \in V_i,h \in H_i} ,  ( \tilde{w}_{h,i}^{ann} )_{h \in H_i}   , ( \chi_v , \beta_v )_{v \in V_i} , [a_i,b_i]   ) ,$$
with $r_i \in \Q$.

$\tilde{Z}$ defines a one parameter family of nice $MC$-cycles $(Z^t)_t$ discontinuous for a finite number of times. The discontinuity at the time $t_0$ is obtained by the formula:
\begin{equation} \label{jump2}
 Z^{t_0^+} - Z^{t_0^-} =  \pm r'  (V', H',  ( w_{v,h}' )_{v \in V',h \in H'} ,  ( \chi_v' , \beta_v' )_{v \in V'}  )
\end{equation}
where
\begin{itemize}
    \item If  $\tilde{w}_{v_1, h_1}^i$ crosses $\tilde{w}_{v_2, h_2}^i$  transversely we have two cases:
    \begin{itemize}
        \item $v_1 \neq v_2$: $H'= H_i \setminus \{ h_1,h_2 \} $  , $V' = V_i \setminus \{v_1, v_2 \}  \sqcup v_0$, where $v_0$ is a new vertex with associated data $\beta_{v_0}'= \beta_{v_1} + \beta_{v_2}$, $\chi_{v_0}'= \chi_{v_1}+ \chi_{v_2} -1 $, $ w_{ v_0,h, t_0}' = w_{ v_1,h,t_0^- }^i + w_{ v_2,h,t_0^- }^i $ for each $h \neq h_1, h_2$. All the other data remain  the same.
\item  $v_1=v_2$:  $H'= H^i \setminus \{ h_1,h_2\} $, $V'=V^i$  , $\chi_{v_1}'= \chi_{v_1}-1$.  All the other data remain the same.
    \end{itemize}
    \item If $\tilde{w}_{v, h_1}^i$ crosses $\tilde{w}^{ann}_{ h_2.i}$, $H'= H_i \setminus \{ h_1,h_2\} $, $V'=V^i$  , $\chi_{v}'= \chi_{v}-1$.  All the other data remain the same.
\end{itemize}
$r'=r_i$ and the sign is defined by the sign of the crossing.


\subsubsection{From MC-cycles to nice-MC-cycles}

We now show how to obtain from a $MC$-cycle $Z$  a nice-MC-cycle $Z^{nice}$. The construction is not unique, but it is unique up to a (small) isotopy. 

\begin{proposition} \label{nonice-nice}

Let $Z$ be a $MC$-cycle.

There exists a $MC$-one-chain $B$ such that $Z + \hat{\partial}B$ is a nice $MC$-cycle.  The $MC$-chain $B$ is not unique, however 
it is compatible with  isotopies in the following sense.

Let $\tilde{Z}$ be an isotopy between two $MC$-cycles $Z_0$ and $Z_1$. 
Let $B_0$ and $B_1$ be $MC$ one chains such that $ Z_0 + \hat{\partial} B_0$  and $ Z_1 + \hat{\partial} B_1$ are nice $MC$-cycles  constructed as in the proof .
There exists  an isotopy of $MC$ one chains    $\tilde{B}$, such that  $\tilde{Z}  + \hat{\partial} \tilde{B}$  an isotopy of nice  $MC$-cycles between $ Z_0 + \hat{\partial} B_0$  and $Z_1 + \hat{\partial} B_1$.

\end{proposition}

\begin{proof}

We use an inductive argument of graphs.
Let $G_0 \in \mathfrak{G}$. Assume that $Z_{\prec G_0} $ is nice.


If $H(G_0) = \emptyset$, set $B_{G_0}=0$.

Now assume $E(G_0) \neq \emptyset$. 


Set $G^{\sharp} = forget_{E(G_0)}G_0$, which is a graph without half edge. 
Denote by $Z^{\diamond, \dagger}_{G^{\sharp}}$ the  linear combination of elements of $ \mathfrak{Gen}(\gamma)^{\dagger}$ determinate by $G^{\sharp}$ in the forgetful property. 
$Z_{G_0, E(G_0)}$ is obtained applying (\ref{nice-graphs}) to $Z^{\diamond, \dagger}_{G^{\sharp}}$. 

Let $Z^{\diamond}_{G^{\sharp}}$  be a linear combination of elements of $ \tilde{\mathfrak{Gen}}(\gamma)$ close in the $C^0$ topology to 
 $Z^{\diamond, \dagger}_{G^{\sharp}}$. Let $Z^{\diamond}_{G_0}$ be the chain obtained applying (\ref{nice-graphs}) to $Z^{\diamond}_{G^{\sharp}}$.  

For each $m$, set
  $$Z_{G_0,m }^{\diamond} = Z_{G_0}^{\diamond}.$$ 
 
We first consider the elements with $|m|=0$.

\begin{lemma}
There exists $B_{G,\{ E_0 \}}$ such that
$$ \partial B_{G,\{ E_0 \}} = Z_{G,\{ E_0 \}} -  Z_{G,\{ E_0 \}}^{\diamond} $$
Moreover $B_{G,\{ E_0 \}}$ can be taken close in the $C^0$-topology to $ Z_{G,\{ E_0,E(G) \}}$.
\end{lemma}
\begin{proof}
From $\partial Z_{G,\{ E_0, E(G) \}} = Z_{G,\{ E_0 \}} -  Z_{G,\{ E(G) \}}$ and the definition of $Z_G^{\diamond}$ follows that there exists   $B_{G,\{ E_0 \}}$ close in the $C^0$-topology to $ Z_{G,\{ E_0,E(G) \}}$ such that   $ \partial B_{G,\{ E_0 \}} = Z_{G,\{ E_0 \}} -  Z_{G}^{\diamond} $. 
\end{proof}

 The first part of the proposition follows from the following lemma.
 
\begin{lemma} \label{contraction}
There exists $ B_{G_0,m}$ transversal to the diagonals associated to $E \setminus E_l$ such that
\begin{equation} \label{B-construction}
 \partial B_{G_0,m} = Z_{G_0,m} - Z_{G_0,m}^{\diamond}  + \sum_i (-1)^i B_{G_0, \partial_i m }
 \end{equation}
Moreover there exists the lift $ \overline{B}_{G_0,m}$.

\end{lemma}

\begin{proof}
We construct $B_{G, m}$ by induction   as a small perturbation of $(-1)^{l+1} (Z^{\diamond}-Z)_{G,  m \sqcup \{E(G) \}}$.

Let $l \in \Z_{\geq 0}$ and assume that we have constructed $B_{G, m}$ for each $m$ with $|m| <l$. 
From the inductive hypothesis it follows that: 
\begin{enumerate}
\item $\partial ( Z_{G_0,m}^{\diamond} - Z_{G_0,m}  + \sum_i (-1)^i B_{G_0, \partial_i m } )=0$
\item  $ Z_{G_0,m}^{\diamond} - Z_{G_0,m}  + \sum_i (-1)^i B_{G_0, \partial_i m } $ is a small perturbation of  $ \partial (Z^{\diamond}-Z)_{G, m \sqcup \{  E(G) \}} $. 
\end{enumerate}
To obtain property $(1)$ apply $\hat{\partial} (Z^{\diamond}-Z) =0$ to $\partial Z_{G,m}$ and use the inductive hypothesis for $B_{G, \partial_i m }$. 
Property $(2)$ follows from the identity 
\begin{equation} \label{Z'-Z}
 \partial (Z^{\diamond}-Z)_{G, m \sqcup \{ E(G) \}} = (Z^{\diamond}-Z)_{G, m} + \sum_i (-1)^i (Z^{\diamond}-Z)_{G,\partial_i m \sqcup \{  E(G) \}} 
 \end{equation}
and  the assumption that $B_{G, \partial_i m  }$ is a small perturbation of $(Z^{\diamond}-Z)_{G, \partial_i m \sqcup \{  E(G) \}} $

From $(1)$ and $(2)$ it  follows that there exists a small perturbation $ B_{G,m}$ of $ Z_{G, m \sqcup \{  E(G) \}}$ such that (\ref{B-construction}) holds.
\end{proof}

Now assume that $\tilde{Z}$ is an isotopy between two $MC$-cycles $Z^0$ and $Z^1$. To construct $\tilde{Z}^{\diamond}$ we use a similar inductive argument as above.

 Assume that $Z^0_{\prec G_0}, Z^1_{\prec G_0} ,\tilde{Z}_{\prec G_0}$ are nice.
 Let $Z^{0 , \diamond}_{G_0}, Z^{1, \diamond}_{ G_0} $ be chains obtained  using the construction above.

As before, set $G^{\sharp} = forget_{E(G_0)}G_0$. 
Denote by $\tilde{Z}^{\diamond, \dagger}_{G^{\sharp}}$ the  linear combination of elements of $ \tilde{\mathfrak{Gen}}(\gamma)^{\dagger}$ determinate by $G^{\sharp}$ in the forgetful property. 
$\tilde{Z}_{G_0, E(G_0)}$ is obtained applying (\ref{nice-graphs}) to $\tilde{Z}^{\diamond, \dagger}_{G^{\sharp}}$.

 Let $\tilde{Z}^{\diamond}_{G^{\sharp}}$  be a linear combination of elements of $ \tilde{\mathfrak{Gen}}(\gamma)$ close to the $C^0$ topology to 
 $\tilde{Z}^{\diamond, \dagger}_{G^{\sharp}}$ with 
 $ \partial \tilde{Z}^{\diamond}_{G^{\sharp}}= \partial \tilde{Z}^{\diamond, \dagger}_{G^{\sharp}}$ and such that 
 for $T \gg 0$
  $$ \tilde{Z}^{\diamond, <-T}_{G^{\sharp}} =  Z^{0 , \diamond}_{G^{\sharp}} \times \R_{<-T}, \thickspace  \tilde{Z}^{\diamond, >T}_{G^{\sharp}} =  Z^{1 , \diamond}_{G^{\sharp}} \times \R_{>T} .$$ 
 
  Let $\tilde{Z}^{\diamond}_{G_0}$ be the chain obtained applying (\ref{nice-graphs}) to $\tilde{Z}^{\diamond}_{G^{\sharp}}$.  
For each $m$ set
  $$\tilde{Z}_{G_0,m }^{\diamond} = \tilde{Z}_{G_0}^{\diamond}.$$



The one chain isotopy $\tilde{B}$ is obtained by induction as before, using the following
\begin{lemma}
 Let $B_{G_0,m}^0$  and $B_{G_0,m}^1$ constructed from $Z^0$ and $Z^1$ as in Lemma \ref{contraction}. 

There exists $\tilde{B}_{G,m}$
such that 
$$ \partial \tilde{B}_{G,m} = \tilde{Z}^{\diamond}_{G,m} -    \tilde{Z}_{G,m}+ \sum_i (-1)^i \tilde{B}_{G,\partial_i m}
$$

$$\tilde{B}^{ <-T}_{G,m } = B^{0 }_{G_0 ,m  } \times (-\infty , -T) $$
$$\tilde{B}^{ >T}_{ G_0 , m} = B^{1 }_{ G_0 ,m} \times (T, \infty )$$
for $T \gg 0$.
\end{lemma}
\begin{proof}
$\tilde{B}_{G, m}$ is constructed as small perturbation of  $(-1)^{l+1} \tilde{Z}_{G, m \sqcup \{ E(G) \}}$    with an inductive argument similar to the one  of Lemma \ref{contraction} . 

\end{proof}

\end{proof}



\subsection{Multi-Link Homomorphism}
We now want to extract rational numbers from $MCH_0(\gamma)$.
For $[\gamma]=0 \in H_1(L,\Q)$ we shall define an homomorphism from $MCH_0(\gamma)$ to the homology of the point as a generalization of of linking number between two curves on $L$ homologicaly trivial.

Given two not intersecting closed one chain $w_1$ and $w_2$ homological trivial on $H_1(L, \Q)$ , denote by $Link(w_1 , w_2) \in \Q$ their linking number.

For $\gamma \in H_1(L, \Z)$ trivial on $H_1(L,\Q)$, define the multi-link map
\begin{equation} \label{multi-link-map}
\text{Multi-Link}:\mathcal{Z}_{\gamma} \rightarrow \Q
\end{equation}
as
$$ (H ,  ( w_h )_h ) \mapsto \sum_{E} \frac{1}{Aut(E)} \prod_{e \in E} Link (w_e,w_e') .$$
The sum is over the set of partitions of $H$ on subsets of cardinality two. We denote $ (w_e,w_e')  = (w_h,w_{h'})$ for $e = \{ h,h'\} \in E$.

\begin{lemma}
The multi-link map (\ref{multi-link-map}) is invariant by isotopies.
\end{lemma}
\begin{proof}
It is easy to check that $Multi-Link(Z^t)$ does not jump under the formula (\ref{jump}).
\end{proof}

\section{Geometric realization}


Fix $X$ a Calabi-Yau symplectic six manifold and $L$ be an oriented Maslov index zero Lagrangian submanifold of $X$. The moduli spaces of bordered pseudo holomorphic curves without marked points with boundary mapped on $L$  have virtual  dimension zero. A spin structure on $L$ defines an orientation of  the moduli spaces of bordered pseudoholomoprhic  curves. 

Assume that $[L]=0 \in H_3(X, \Z)$ and  fix a four chain $K$ such that $\partial K =L$.

 From $K$ we can define an homology class  :
\begin{equation} \label{chern-four-chain}
\mathfrak{c}(K) = pr_*[(\partial T) \cap K - r ((\partial T) \cap K)]  \in H_1(L, \Z),
\end{equation}
where $T$ is a small tubular neighborhood of $L$, $r : \partial T \rightarrow \partial T$ is the reverse map, 
 $pr: \partial  T \rightarrow L$ is the projection. In (\ref{chern-four-chain}), before apply the projection,  we have used the isomorphism $H_3(\partial T, \Z)^- = H_1(L, \Z)$, where $H_3(\partial T, \Z)^-$ is the vector space of the $r^*$-anti-invariant elements of $H_3(\partial T, \Z)$.






In this section we are going to construct a $MC$-cycle from the moduli space of bordered pseudoholomorphic curves. The data defining the set of decorated graphs are:
\begin{itemize}
\item $\Gamma= H_2(M,L,\Z)$
\item $\partial: \Gamma \rightarrow H_1(L,\Z)$
is the usual boundary in homology.
\item $\omega(\beta) = \int_{\beta} \omega$
for $\beta \in \Gamma$.
\item $\mathfrak{c} = \mathfrak{c}(K) \in H_1(L, \Z)$ is defined using (\ref{chern-four-chain}).
\end{itemize}



The following lemma is standard and it is a simple application of the isoperimetric inequality.
\begin{lemma}
There exists a constant $C^{supp}$ such that    
$$ \parallel \beta \parallel \leq C^{supp} \omega(\beta) $$
 for each class $\beta \in \Gamma$ for which there exists a pseudo-holomorphic curve in class  $ \beta$.
 \end{lemma}
In definition of  graphs $ \mathfrak{G}(\beta,\chi)$ given in (\ref{support})  we fix the constant $C^{supp}$ big enough such that the Lemma above holds.




\subsection{Kuranishi Spaces}

Fix a compact metrizable topological space $\mathcal{M}$. In general we are interested to moduli space $\mathcal{M}$ that is singular. The Kuranishi structure formalism allows us to perturb $\mathcal{M}$ in an abstract way to obtain a smooth moduli space in a suitable sense.
In this sub-section we collect the basic facts about Kuranishi Spaces we shall need. For details we refer to \cite{FO}, or appendix A of \cite{FO3}.

\subsubsection{Kuranishi Structures}

\begin{definition}
For $p \in \mathcal{M}$, a $\mathbf{Kuranishi-neighborhood}$ of $p$ is a quintet $(V_p,E_p,\Gamma_p,s_p, \psi_p)$ where 
\begin{itemize}
\item $V_p$ is a smooth finite dimensional manifold (which may have boundary or corners), 
\item $E_p \rightarrow V_p$ is a vector bundle over $V_p$, 
\item $\Gamma_p$ is a finite group which acts smoothly on $V_p$ and acts compatibly on $E_p$, 
\item $s_p: V_p \rightarrow E_p$ is a $\Gamma_p$-equivariant smooth section, 
\item $\psi_p$ is a homeomorphism from $s_p^{-1}(0)/ \Gamma_p$ to a neighborhood of $p$ in $X$. 
\end{itemize} 
$E_p$ is called the \emph{obstruction bundle} and $s_p$ the \emph{Kuranishi map}.
\end{definition}

A \emph{coordinate changes} between two Kuranishi charts $(V_p,E_p,\Gamma_p,s_p, \psi_p)$, $(V_q,E_q,\Gamma_q,s_q, \psi_q)$, with $q \in \psi_p(s^{-1}(0)/ \Gamma_p)$ is given by a triple $( \hat{\phi}_{pq}, \phi_{pq}, h_{pq})$, where $h_{pq}: \Gamma_p \rightarrow \Gamma_q$ is an injective homomorphism ,  $( \hat{\phi}_{pq}, \phi_{pq}): E_p \times V_{pq} \rightarrow E_q \times V_p $ is an $h_{pq}$-equivariant smooth embedding of vector bundles, where $ V_{pq} $ is a $\Gamma_q$ invariant open subset of $V_q$ such that $o_q \in  V_{pq}$ with $s_q (o_q)=0$ and $\psi_q(o_q) =q$ (see Appendix A of \cite{FO3} for the precise definition) ;

A Kuranishi structure on $\mathcal{M}$ assigns a Kuranishi neighborhood $(V_p,E_p,\Gamma_p,s_p, \psi_p)$ to each $p \in \mathcal{M}$, and for each $q \in \psi_p(s^{-1}(0)/ \Gamma_p)$ a coordinate changes $( \hat{\phi}_{pq}, \phi_{pq}, h_{pq})$. The integer $\text{dim}{E_p}- \text{dim}{V_p}$ is independent of $p$ and is called the \emph{virtual dimension} of $\mathcal{M}$. The coordinate changes are required to satisfy some obvius compatibility condiction (see Appendix A of \cite{FO3} for the precise definition).

We assume that the Kuranishi structure has a \emph{tangent bundle}. This means that the natural map inducted by the fiber derivative of $s_p$  
\begin{equation} \label{tangent}
d_{fiber} s_p : N_{\phi_{pq}(V_q)}(V_p) \rightarrow \frac{E_q|_{\text{Im}(\phi_{pq})}}{\hat{\phi}_{pq}(E_q)} 
\end{equation}
is an isomorphism. Here $N_{\phi_{pq}(V_q)}(V_p)$ is the normal bundle of $\phi_{pq}(V_q)$ in $V_p$.

An \emph{orientation} of a Kuranishi Structure is an orientation of $E_p^* \oplus TV_p $ for every $p$ that is compatible with (\ref{tangent}).


\subsubsection{Fiber Product}
A map $f : \mathcal{M} \rightarrow Y$ between $\mathcal{M}$ and an orbifold $Y$ is called \emph{strongly smooth}  if, in each Kuranishi neighborhood $(V_p,E_p,\Gamma_p,s_p, \psi_p)$, $f$ can be represented by a smooth map $f_p: U_p \rightarrow Y$. The maps $f_p$ are required to fulfill some obvious compatibility conditions. The map $f$ is called \emph{weakly submersive} if each $f_p$ is a submersion.

Let $f : \mathcal{M} \rightarrow Y$ be a strongly continuous and weakly submersive map, let $W$ be a manifold and $g:  W \rightarrow Y$ be a smooth immersion. 
The space
\begin{eqnarray} \label{fiber-product}
\mathcal{M} \times_Y W = \{(p,q) \in \mathcal{M} \times W | f(p) = g(q)  \} 
\end{eqnarray}
has a Kuranishi structure as follows. For each $(p,q)  \in \mathcal{M} \times_Y W $ let $(V_p,E_p,\Gamma_p,s_p, \psi_p)$ be a Kuranishi neighborhood of $p$. Define a Kuranishi neighborhood $(V_{(p,q)},E_{(p,q)},\Gamma_{(p,q)},s_{(p,q)}, \psi_{(p,q)})$ of $(p,q)$, where $V_{(p,q)} = V_p \times_W Y$ and $E_{(p,q)},\Gamma_{(p,q)},s_{(p,q)}, \psi_{(p,q)}$ are naturally defined from $E_p,\Gamma_p,s_p, \psi_p$ (see \cite{FO} or section $A1.2$ of \cite{FO3}).
If $\mathcal{M}$, $Y$ and $W$ are oriented, $\mathcal{M} \times_Y W$ is oriented by the relation 
$ E_p^* \oplus  TV_p =   E_{(p,q)}^* \otimes   T_*(V_{(p,q)})   \otimes   N_W Y$
where $N_W Y$ is the normal bundle of $W$ in $Y$. We write symbolically  
$$   T_*\mathcal{M}  =   N_W Y   \oplus  T_*(\mathcal{M} \times_W Y)     .$$



\subsubsection{Perturbation}

In order to achieve transversality it is necessary to introduce \emph{multisections}.   
A $n$-multisection $\mathfrak{s}$ of the orbibundle $E \rightarrow V$ is a continuous, $\Gamma$-equivariant section of the bundle $S^n E \rightarrow V$, where  $S^n E \rightarrow V$ is the quotient of the vector bundle $E^n \rightarrow V$ by the symmetric group $S_n$. We shall consider only liftable multisections, that is we require  that there exists $\tilde{s}= (s_1, . . . , s_n) : V \rightarrow E^n$ with each $s_i$ continuous such that $s = \pi \circ \tilde{s}$, where $\pi : E^n \rightarrow S^n E$ is the projection. 

To a map $f:\mathcal{M} \rightarrow Y$ strongly smooth map in \cite{FO} it is associated \emph{virtual chain}, that is a smooth simplicial chain on $Y$ of dimension equal to the virtual dimension of $\mathcal{M}$. In order to construct the virtual chain we need to pick a \emph{perturbation data} $\mathfrak{s}$. This involves the choice of \emph{good coordinate system}  and a smooth transverse multisections sufficiently close in $C^0$ to the starting Kuranishi map $s$.  
The transversality require that $\mathfrak{s}$ is transverse on each stratum of the boundary. 
By sufficiently close we need to ensure that the perturbed Kuranishi space remain compact (see \cite{FO}). We denote the virtual chain associate to the perturbation data $\mathfrak{s}$ symbolically as $f_*(\mathfrak{s}^{-1}(0))$.

In \cite{FO} is proved the existence of the perturbation of the Kuranishi structure. We state their result  as in Theorem $A1.23$ of \cite{FO3}:
\begin{lemma} (\cite{FO}) \label{perturbation-lemma-FO}
Suppose that the Kuranishi structure over $\mathcal{M}$ has a tangent bundle. There exist a family of transversal multisections  $\mathfrak{s}_{\epsilon}'$ such that it converges to $s$.

Moreover if $K \subset \mathcal{M}$ is compact subset and $\mathfrak{s}_{\epsilon}''$ is a family of transversal multisections in a Kuranishi neighborhood of $K$ that converges to $s$ then we can take the family  $\mathfrak{s}_{\epsilon}'$ so that it coincides with $\mathfrak{s}_{\epsilon}''$ in a Kuranishi neighborhood of $K$.
\end{lemma}





\subsection{Moduli Space of pseudo-holomorphic bordered curves}

Let $X$ be a symplectic Calabi-Yau $6$-manifold and $L$ a Lagrangian submanifold of Maslov index zero.
Fix a almost-complex structure $J$ on $X$ compatible with the symplectic structure.
In this subsection we review the basic results about moduli space of $J$ psudoholomorphic maps on $X$ with boundary mapped on  $L$. 



Recall that a semi-stable map 
$(\Sigma,u) \in \overline{\mathcal{M}}_{(g,h), n, \vec{m}}(\beta)$ 
is defined by:
\begin{itemize}
\item A topological space $\Sigma$ which is union of Riemmanien surfaces with boundary, which are called irreducible components. The intersection of two components is a finite set of points, called singular points.  The intersection of three different components is empty. A singular point can be or internal or boundary. 
\item A continuous map $u: \Sigma \rightarrow X$, which is $J$-holomorphic on each of the components, such that $u(\partial \Sigma) \subset L$.
\end{itemize}

We write the decomposition in irriducible components of $(\Sigma,u)$ as
\begin{equation} \label{stable-decomposition}
(\Sigma,u) = \bigsqcup_{a \in A} (\Sigma_a , u_a) .
\end{equation}
The singular and marked points of $\Sigma_a$ are called special points of  $\Sigma_a$.

The component $(\Sigma_a,u_a)$ is stable if $u_a$ is not constant or $\Sigma_a$  satisfies the following condition: 
\begin{enumerate}
\item if $\Sigma_a$ is a disk, then  $\Sigma_a$ has at least $3$ boundary special points or $1$ internal and $1$ boundary special point;
\item if $\Sigma_a$ is an annulus, then  $\Sigma_a$ has at least $1$ internal or boundary special point;
\item if $\Sigma_a$ is a sphere, then  $\Sigma_a$ has at least $3$ internal special point;
\item if $\Sigma_a$ is a torus, then  $\Sigma_a$ has at least $1$ internal special point.
\end{enumerate}
The semi-stable map $(\Sigma,u)$ is said stable if $ (\Sigma_a , u_a)$ is stable for each $a \in A$.

Denote by $\overline{\mathcal{M}}_{(g,h), n, \vec{m}}(\beta)$ the moduli space of stable pseudoholomorphic maps with genus $g$, $h$ boundary components, $n$ internal marked points, $\vec{m}=(m_1,m_2,...,m_h)$ number of boundary marked points, in the homology class $\beta \in H_2(M,L)$. As generalization of the Deligne-Manford topology for the moduli of stable curves, $\overline{\mathcal{M}}_{(g,h), n, \vec{m}}(\beta)$ admits a natural topology which is Hausdorff and compact.






Let $\chi= 2-2g-h$ be the Euler Characteristic of the the surface. We introduce a partial order on the type $(\beta, g,h,\overrightarrow{m})$. We define $(\beta', g',h', n', \overrightarrow{m}') \prec (\beta, g,h, n , \overrightarrow{m})$ if $(\omega(\beta'), 2 n'+ \sum_{b'} m_b'-\chi') < (\omega(\beta),  2 n+ \sum_b m_b-\chi)$ in lexicographic order.
The decomposition (\ref{stable-decomposition}) induces a natural stratification of  $\overline{\mathcal{M}}_{(g,h), n, \vec{m}}(\beta)$. The stable maps with $k$ internal singular points and $l$ boundary singular points belongs to the stratum of real codimension $2k+l$. The lower strata can be described in terms of fiber products of moduli of curve of smaller type.

Set 
$$\kappa= 2n- \sum_i|m_i| - \chi  .$$


\begin{lemma} (\cite{L})  \label{argument}
The space $\overline{\mathcal{M}}^{main}_{(g,h),(n, \overrightarrow{m})}(\beta)$  can be endowed with a  Kuranishi structure with corners (natural up to equivalence) invariant by cyclic permutation of boundary marked points and such that the evaluation map on the internal or boundary marked points 
\begin{equation} \label{evaluation map}
\text{ev} :\overline{\mathcal{M}}^{main}_{(g,h),(n, \overrightarrow{m})}(\beta) \rightarrow X^n \times L^{\overrightarrow{m}}
\end{equation}
is strongly smooth and weakly submersive. Moreover on the lower strata  it is compatible with the Kuranishi structure induced by the fiber product.  
\end{lemma}


\begin{proof} We sketch the general argument for the construction of the Kuranishi structure associated to the moduli space of pseudo-holomorphic curves.  For the details see \cite{L}, \cite{FO3}, that is essentially similar to the closed case considered in \cite{FO}.
 

The construction of the obstruction bundle is made inductively on neighborhoods of the strata, beggining with the strata with maximal number of components.

Assume we have defined the Kuranishi structure for moduli spaces of type smaller than $ (\beta, g,h,\overrightarrow{m})$. 
We first construct a preliminary definition of Kuranishi neighborhood $V_u'$  for each $(\Sigma,u) \in \overline{\mathcal{M}}^{main}_{(g,h),(n, \overrightarrow{m})}(\beta)$. $V_u'$ is defined as the product of three spaces $V_u'=V_{deform,u} \times V_{resolve,u} \times V_{map,u}$. $V_{deform,u}$ is associated to the deformations of $\Sigma$, $ V_{resolve,u}$ is associated to the resolutions of the singularities of $\Sigma$,  $V_{map,u}$ is associated to the deformations of $u$ and it is related to the obstruction bundle as follows.

The linearization of   $\overline{\partial}_J$ defines a Fredholm  operator between Banach spaces 
$$D(\overline{\partial}_J):  W^{1,p} (\Sigma, u^*(TM); u^*(TL)) \rightarrow  L^p (\Sigma, u^*(TM) \otimes \Lambda^{0,1}(\Sigma)).$$
We fix $p$ big enough such that the elements of  $W^{1,p} (\Sigma, u^*(TM); u^*(TL))$ are continuous functions . (in the case $\Sigma$ are more components as in (\ref{stable-decomposition}), the space  $W^{1,p} (\Sigma, u^*(TM); u^*(TL))$ is defined as the direct sum of the spaces associated to each component with an obvious matching condition on the singular points.)
Let $E_u' \subset L^p (\Sigma, u^*(TM) \otimes  \Lambda^{0,1}(\Sigma))$ be a finite dimension sub-vector space such that 
$$D(\overline{\partial}_J):  W^{1,p} (\Sigma, u^*(TM); u^*(TL)) \rightarrow  L^p (\Sigma, u^*(TM) \otimes  \Lambda^{0,1}(\Sigma))/E_u'$$ 
is surjective.  
We assume that $E_u'$ is $\text{Aut}(\Sigma,u)$ invariant, its elements are smooth with support away from the boundary and internal marked points. 

The set $V_{map}= \{v : \Sigma \rightarrow M | \overline{\partial}_J(v) \in E_u'   \} $ (here we are identifying $E_u'$ with a subspace of $W^{1,p} (\Sigma, v^*(TM); v^*(TL))$ using parallel transport, for $v$ suitable close to $u$) is a smooth finite dimension manifold with corners. 
We can choice $E_u'$ such that the evaluation map is weakly submersive.

On a point $u$ of the lower stratum  with more than one component (as in \ref{stable-decomposition}) , the  Kuranishi neighborhood $V_u'$ is constructed using a gluing argument from the fiber product of the Kuranashi neighborhood  associated to each $\Sigma_a$, using the Kuranishi structures defined on the moduli of smaller type. The obstruction bundle is the direct sum  the obstruction bundles of each component.  In the case $\Sigma$ has unstable components we replace $V_{map}$ by the orthogonal complement of  $Lie(\text{Aut}(\Sigma))$ in  $V_{map}$.




To actually define a Kuranishi structure we need to modify the obstruction bundles in order to allow for coordinate changes. Assume that $V_u'$ is closed, and let $u_1,u_2,..., u_N$ be a finite number of points of $\overline{\mathcal{M}}^{main}_{(g,h),(n, \overrightarrow{m})}(\beta)$ such that $ \bigcup_i  \mathring{V}_{u_i}' = \overline{\mathcal{M}}^{main}_{(g,h),(n, \overrightarrow{m})}(\beta) $. For $u \in \overline{\mathcal{M}}^{main}_{(g,h),(n, \overrightarrow{m})}(\beta)$, take 
\begin{equation} \label{obstruction-modification}
 E_u = \bigoplus_{i | u \in  {V}_{u_i}'}  E'_{u_i}
\end{equation}
and use $E_u$ to define $V_u$ as before. 
In order to make sense  (\ref{obstruction-modification}) we need to identify $E'_{u_i} $ with a subspace of $L^p (\Sigma, u^*(TM) \otimes  \Lambda^{0,1}(\Sigma))$. This involve in particular the choice of a biholomorphism between $\Sigma$ and an element of $V_{deform,u_i} \times V_{resolve,u_i}$.
In the case $\Sigma_i$ have unstable components the set of these biholomorphims come in families which can be locally parametrized by a  neighborhood of zero in  $Lie(\text{Aut})(\Sigma_i) $.  Since $\text{Aut}(\Sigma)$ is not a finite group, this causes a problem since $E_u'$ cannot be taken $\text{Aut}(\Sigma)$ invariant. A way to get around this problem is to use the element of $V_{deform} \times V_{resolve}$ which minimize in a suitable sense the distance between $u$ and $u_i$ (see \cite{FO})  . 

The modification (\ref{obstruction-modification}) can be made such that the Kuranishi structure on the lower strata is not modified, as follows. Take a finite set of $u_i$ in the lower strata such that the complementary of $\sqcup V_{u_i}$ is compact. Then we take a finite number of $V_{u_j}'$ which do not  intersect the lower strata and such that $ \bigsqcup_i  \mathring{V}_{u_i}' \sqcup  \bigsqcup_j  \mathring{V}_{u_j}' = \overline{\mathcal{M}}^{main}_{(g,h),(n, \overrightarrow{m})}(\beta) $. Use this covering to implement (\ref{obstruction-modification}). 

Finally, we can make the Kuranishi structure cyclic symmetric adding to the obstruction bundle its images by the cyclic symmetry.
If the Kuranishi structure was cyclic symmetric on the moduli spaces of smaller type, this step does not modify the Kuranishi structure of the lower strata.


\end{proof}

\begin{lemma} (\cite{L}) \label{orientation-curves}
A spin structure on $L$ and an ordering of boundary marked points induce an orientation of $\overline{\mathcal{M}}^{main}_{(g,h),(n, \overrightarrow{m})}(\beta)$. 
\end{lemma}





We now introduce  some notation which will be useful on the paper.

Denote by $  \overline{\mathcal{M}}_{(g,V, D, (H)_v)}(\beta) $ 
the moduli  of curves of genus $g$ , homology class $\beta$, whose boundary components are labelled by $V$, internal marked points are  labelled by $D$ and boundary  marked points are labelled by $H_v$ for each $v$.


After the choice of an order for $ V, D$ and  $H_v$ , the moduli $  \overline{\mathcal{M}}_{(g,V, D, (H)_v)}(\beta) $  can be identified  with
 $ \overline{\mathcal{M}}^{main}_{(g,h),(n, \overrightarrow{m})}(\beta)$ , where $h= |V|$, $n= |D|$
$m_i = |H_i|$.

A Kuranishi structure on $  \overline{\mathcal{M}}_{(g,V, D, (H)_v)}(\beta) $ is defined from 
a Kuranishi structure on $ \overline{\mathcal{M}}^{main}_{(g,h),(n, \overrightarrow{m})}(\beta)$ which is compatible with the cyclic symmetry of the boundary marked points, the permutation of boundary components and the permutations of internal marked points .




\subsubsection{Boundary Strata}

Let $\beta \in H_2(X,L)$, $g \in \Z_{\geq 0}$, $V$ and $D$ finite sets and, for each $v \in V$, $H_v$ a finite cyclic order set.
In this subsection we describe the corner faces of $\overline{\mathcal{M}}_{(g, D,V, (H_v)_v) }(\beta)$.


Let $(\Sigma,u) \in   \overline{\mathcal{M}}_{(g,V,D, (H_v)_v) }(\beta)$ with the decomposition in stable components given as in (\ref{stable-decomposition}).  Consider the equivalence relation on $A$ generated by $a \sim a'$ if $\Sigma_a $ and $\Sigma_{a'} $ are connected by an internal singular point. Let $B$ be the set of the equivalence classes and for $b \in B$ let $\Sigma_b = \sqcup_{a \in b} \Sigma_a $. Thus, we have the decomposition of $(\Sigma,u)$  
\begin{equation} \label{stable-decomposition2}
\Sigma = \bigsqcup_{b \in B} \Sigma_b 
\end{equation}
where now the components $\Sigma_b$ are connected only by boundary singular points. The decomposition (\ref{stable-decomposition2}) can be described in terms of the decorated graphs of subsection \ref{section-graphs}, where the boundary nodal points of $\Sigma$ correspond to internal edges of the graph. We write the formula explicitly as follows.


Let $(\beta,(g, D,V, (H_v)_v) )$ be the decorated graph with one only component of genus $g$ , topological charge $\beta$,  vertices $V$, half edges $(H_v)_v$, degenerate vertices $D$.  Let
$\mathfrak{G}(\beta,(g, D,V, (H_v)_v) )$ be the set of graphs as in  (\ref{sub-graphs}) with $G= (\beta,(g, D,V, (H_v)_v) )$.




As a consequence of Lemma \ref{argument} we have a bijection 
$$ \mathfrak{G}(\beta,(g, D,V, (H_v)_v) ) \leftrightarrow  \{    \text{   corner faces of  }   \overline{\mathcal{M}}_{(g, D,V, (H_v)_v)}(\beta)  \} .$$
Denote by  $ \overline{\mathcal{M}}_{(g, D,V, (H_v)_v)}(\beta)(G',E') $ the corner face corresponding to $(G',E') \in \mathfrak{G}(\beta,(g, D,V, (H_v)_v) )$.   

Write $E' = E^{in}(G') \sqcup D'$, with $D' \subset D(G')$.    $   \overline{\mathcal{M}}_{(g, D,V, (H_v)_v)}(\beta)(G',E') $  comes  with an identification  of Kuranishi spaces
\begin{multline} \label{corner-faces}
    \overline{\mathcal{M}}_{(g, D,V, (H_v)_v)}(\beta)(G',E') \cong  \\
    \left( \prod_{c \in Comp(G')}   \overline{\mathcal{M}}_{(g_c, D_c,V_c, (H_v)_{v \in V_c})}(\beta_c) \times_{L^{H(G')} \times X^{D'}} ( (Diag)^{E^{in}(G')} \times L^{D'}) \right) / \text{Aut}(G',E') ,
\end{multline}
where $Diag$ is the diagonal of $L \times L$.


\subsubsection{Forgetful Compatibility} \label{forgetful}


In this subsection we  define the notion of forgetful compatibility for the Kuranishi structures of the moduli space of pseudo-holomorphic curves.


Fix $(\Sigma,u) \in \overline{\mathcal{M}}_{(g,V,D, (H_v)_v) }(\beta)$. Let $h \in H_{v_0}$ and set $H_v' = H_v$ if $v \neq v_0$,
$ H_{v_0}' = H_{v_0} \setminus \{ h\}   $.  

Define $\mathbf{forget}_h (\Sigma,u) \in \overline{\mathcal{M}}_{(g,V,D, (H_v')_v) }(\beta)$ as follows.
Let $ (\Sigma,u) = \sqcup_{a \in A} (\Sigma_a,u_a) $  be 
the decomposition of $(\Sigma,u) $ in irreducible components as in (\ref{stable-decomposition}).
Let $(\Sigma_0,u_0)$ be the component to which $h$ is attached. If $(\Sigma_0,u_0)$ is stable after removing $h$, $ \mathbf{forget}_h (\Sigma,u)$ is given by removing $h$ from $\Sigma_0$. Otherwise, $u$ is constant on $\Sigma_0$ and we have the following cases:
\begin{enumerate}
\item $\Sigma_0$ is a disk with two boundary marked points, one boundary singular point, and no special internal points; 
\item $\Sigma_0$ is a disk with one boundary marked point, two boundary singular points, and no special internal points;  
\item  $\Sigma_0$  is a disk with one internal singular point and no other special points rather than $h$. 
\end{enumerate}

In case $(1)$,  let $h'$ be the boundary marked point different from $h$.  Remove $\Sigma_0$ and rename the singular boundary point of $\Sigma_0$ to $h'$. 

In case $(2)$,  $h$ is the only boundary marked point. Remove $\Sigma_0$ and identify the two singular boundary points of $\Sigma_0$. 

In case $(3)$ remove $\Sigma_0$ and transform the singular internal point of $\Sigma_0$ in an internal marked point. This gives a boundary node of type $E$.







Let  $p=(\Sigma,u) \in  \overline{\mathcal{M}}_{(g,V,D, (H_v)_v) }(\beta) $ and  $p'= \mathbf{forget}_h (\Sigma,u) \in \overline{\mathcal{M}}_{(g,V,D, (H_v')_v) }(\beta)$. 
A  Kuranishi neighbourhood  $(V_p,E_p,\Gamma_p,s_p, \psi_p)$ of $p$ is forgetful compatible with the  Kuranishi neighbourhood $(V_{p'},E_{p'},\Gamma_{p'},s_{p'}, \psi_{p'})$ of $p'$, if the following happen.

Assume first that $p'$ is not in the case $(3)$ above.  Let $F$ be an open interval if $H_{b_0}' \neq \emptyset$, or the one dimension circle if  $H_{b_0}' = \emptyset$.
Then  ${V}_p = V_{p'} \times F $, ${E}_p = E_{p'} \times F $, ${\Gamma}_p= \Gamma_{p'}$, ${s}_{p'}(x, t) = (s_p(x), t)$. 



Assume now that $p'$ is the case $(3)$ above, i.e., $(\Sigma',u')$ is in the boundary face of a node of type $E$. 
We can assume that $V_{p'}$ factorize as $V_{p'}= U' \times [0, \epsilon)$, for some open set $U'$ and $\epsilon >0$ small. Accordingly to this we assume that  $E_{p'} = E_{p'}' \times [0, \epsilon)$, $s_{p'}= s_{p'}' \times [0, \epsilon) $ and 
$ \Gamma_{p'}$ acts trivially on the $[0, \epsilon)$ factor.
Then $V = U' \times D^2_{\epsilon}  $, where $D^2_{\epsilon}$ is the disk of radius $\epsilon$. The forget map $V_{p} \rightarrow V_{p'}$ is given by the product of the identity map of $U'$  and the submersion $   D^2_{\epsilon} = \frac{[0, \epsilon) \times S^1}{{0} \times S^1} \rightarrow  [0, \epsilon)$. 
Accordingly to this we have $E_{p}= E_{p'}'\times D^2_{\epsilon} $, $s_{p}= s_{p'}'\times D^2_{\epsilon}$ and
$ \Gamma_{p} = \Gamma_{p}'$ acts trivially on the $ D^2_{\epsilon}$ factor.

\begin{definition}
We say that the boundary puncture $h$ is $\mathbf{forgetful-compatible}$ if the Kuranishi structures 
$ \overline{\mathcal{M}}_{(g,V,D, (H_v)_v) }(\beta)$ and $ \overline{\mathcal{M}}_{(g,V,D, (H_v')_v) }(\beta)$
are related as above.
\end{definition}

\subsection{Decorated Graphs} \label{section-graphs}
A decorated graph $G$ consists in an array 
$$(Comp, (V_c, D_c, \beta_c ,g_c )_c, (H_v)_v, E)$$ 
where 
\begin{itemize}
\item  $Comp(G)$ is a finite set, called set of components of $G$;
\item For each $c \in Comp(G)$
\begin{itemize}
\item  $V_c$ is a finite set, the set of boundary components or vertices of $c$;
\item  $D_c$ is a finite set , the set  of internal marked points of $c$;
\item  $\beta_c \in \Gamma$, called charge of $c$
\item   $g_c \in \Z_{\geq 0}$, called  genus of $c$.  
\end{itemize}
Set 
$$\beta(G) :=  \sum_{c \in Comp(G)} \beta_c \in \Gamma , V(G) := \bigsqcup_{c \in Comp(G) } V_c , D(G):= \bigsqcup_{c \in Comp(G)} D_c .$$
\item For each $v \in V(G)$, $H_v$ is a cyclic ordered finite set. Set $H(G)= \sqcup_{v \in V_G} H_v$. $H(G)$ is called  the set of half-edges of $G$;
\item  $E$ is a partition of  $H(G)$ in sets of cardinality one or two, whose elements are called edges of $G$. The elements  of cardinality two are called internal edges $E^{in}(G)$, the elements of cardinality one external edges $E^{ex}(G)$;
\end{itemize}


We assume that
$$ \beta_c \in \Gamma_{tors} \Rightarrow \beta_c=0 $$


A component $c$ is called unstable if $\beta_c = 0$ and $2 \chi_c - | H_c | \geq 0 $.
The graph $G$ is called stable if all its components are stable.

Fix a norm $ \parallel \bullet \parallel$ on $\Gamma_{\R}= \Gamma \otimes  \R$.  For each positive real number $C^{supp} \in \R_{>0}$  denote by $ \mathcal{G}(\beta, \kappa,C^{supp})$ the set of stable decorated graphs with topological charge $\beta$ with $|E^{ex}(G)| - \chi(G)= \kappa$ and
 \begin{equation} \label{support-not-ab}
\parallel \beta_c \parallel \leq C^{supp}  \omega(\beta_c)
\end{equation}
 for each  $ c \in Comp(G)$.

Observe that $\mathfrak{G}(\beta, \kappa, C^{supp}) $ is a \emph{finite} set.
In the following of this section we fix the constant $C^{supp}$ and we omit the dependence on  $ \parallel \bullet \parallel$ and $C^{supp} $ in the notation.

In the next subsection, for each $e \in E^{in}(G) \sqcup D(G)$ it is defined the graph $\delta_e G$.  The operation $\delta_e$ associated to different edges commute:
$$   \delta_{e_1} \circ \delta_{e_2} G =  \delta_{e_2} \circ \delta_{e_1}  G \text{   for each   }  e_1 , e_2 \in E^{in}(G) \sqcup D(G).$$
Given a set of edges $\{e_1,e_2,...,e_n\} \subset E^{in}_G \sqcup D(G)$ we denote by $G/\{ e_1,e_2,...,e_n \} $ the graph obtained applying each $\delta_{e_i}$ to $G$: 
$$  G/\{e_1,e_2,...,e_n\} =  \delta_{e_1} \circ \delta_{e_2} \circ ...  \circ \delta_{e_n} (G) . $$

Define  the decorated graph
$$ \Sigma_G =  G / \{ E^{in}(G ) \sqcup D(G) \} .$$ 
We have
$$D(\Sigma_G)= \emptyset, \quad E^{in}(G)= \emptyset  , \quad \beta(\Sigma_G)= \beta(G), \quad E^{ex}(\Sigma_G)  = E^{ex}(G).$$
We can identify $\Sigma_G$ with a (not necessarily connected) surfaces with boundary marked points $E^{ex}(G)$. 
We call $g(G)= g(\Sigma_G )$ the genus of $G$ and $h(G) = |V(\Sigma_G)| $ the number of boundary components of $G$.

For a decorated graph $G$ set
\begin{equation} \label{sub-graphs}
 \mathfrak{G}(G) =\{  (G',E') | G' \in \mathfrak{G},  E' \subset E^{in}(G') \sqcup D(G'),  G'/E' \cong G \}/\sim .
\end{equation}





\subsubsection{Operation $\delta_e$}
To a graph $G$ and $e \in E^{in}(G) \sqcup D(G)$ it is associated a graph $\delta_e G$  as follows.

We first consider the case $e \in D(G)$.
Let $c_0 \in Comp(G)$ be the component such that $e \in E_{c_0}$.  $\delta_e G$ is defined discarding $e$ from $E_c$ and adding to $V_{c}$ a new vertex $v_e$, with $H_{v_e} = \emptyset$. 
 All the other data defining $G$ stay the same.

We now consider the case $ e \in E^{in}(G)$
Let $e=\{ h_1,h_2 \} \in E^{in}(G)$ be an internal edge of $G$. 
We have different cases:
\begin{itemize}
\item
Assume $h_1 \in H_{v_1}$ and $h_2 \in H_{v_2}$ for $v_1 , v_2 \in V(G)$ with $v_1 \neq v_2$.  Define ordered sets $I_1$ and $I_2$ such that $H_{v_1 }= \{ h_1 , I_1  \}$ and  $H_{v_2 }= \{ h_2 , I_2  \}$ as cyclic ordered sets. $V(\delta_e G)$ is defined by replacing in $V(G)$ the vertices $v_1$ and $v_2$   by a unique vertex $v_0$, with $ H_{v_0} =  \{ I_1 , I_2  \} .$

\begin{itemize}
\item If  $v_1 \in V_{c_1}$, $v_2 \in V_{c_2}$ for  $c_1 , c_2 \in Comp(G)$ with $c_1 \neq c_2$, $Comp(\delta_e G)$ is obtained replacing in $Comp(G)$ the components $c_1$ and $c_2$ with a unique component $c_0$. $V_{c_0}$ is obtained by $V_{c_1} \sqcup V_{c_2} $ replacing the vertices $v_1$ and $v_2$ with $v_0$. $E_{c_0}=E_{c_1} \sqcup E_{c_2}$, $  g_{c_0} = g_{c_1}+ g_{c_2}$.
\item If  $v_1,v_2 \in V_{c_0}$ for some  $c_0  \in Comp(G)$, set $Comp(\delta_e G) =Comp(G)$ with the genus $g_{c_0}$ increased by one and all the other data of $c_0$ remain the same. 
\end{itemize}
\item
Assume $h_1,h_2 \in H_{v_0}$ for $v_0 \in V(G)$. Write the cyclic order set $H_{v_0 }$ as $H_{v_0 }= \{ h_1 , I_1 , h_2, I_2 \}$ for some order sets $I_1$ and $I_2$.  $V(\delta_e G)$ is given by $V(G)$ replacing $v_0$ by two vertices $v_0', v_0''$ with
$ H_{v_0'} =  I_1 \text{    }  H_{v_0''} =  I_2 .$

Set $Comp(\delta_e G) = Comp(G)$. Let $c_0 \in Comp(G)$ such that $v_0 \in Comp_{c_0}$. $V_{c_0}$ in $\delta_e G$ is obtained replacing $v_0$ with $v_0', v_0''$,   and all the other data of $c_0$ remain the same. 
\end{itemize}







\subsection{Moduli Space of Multi-Curves} \label{Moduli-interp-section}


To a graph $G \in \mathfrak{G}(\beta,g,V,D,(H_v)_v)$ we associate the moduli space of multi-curves
\begin{equation}    \label{multi-curve0}
 \overline{\mathcal{M}}_{G} :=    \left( \prod_{c \in C(G)} \overline{\mathcal{M}}_{g_c, V_c, D_c, (H_v)_{v \in V_c} }(\beta_c)  \right)  / \text{Aut}(G) 
 \end{equation} 
where the boundary marked points which belong to $E^{ex}(G)$ are considered forgetful compatible, and all the other boundary marked points are weakly submersive. 

In order to construct a $MC$-cycle from the moduli space of multi-curves it is necessary to extend the definition to decorated graphs with marked edges and define suitable constraints on 
 the associated moduli space.  
In these moduli spaces the forgetful compatible and weakly submersive Kuranishi structures can be considered at the same time. This fact makes the construction of the $MC$-cycle relatively easy and it has not an analogous  when we work with the moduli space of pseudo-holomorphic curves.


 



Define the set of  decorated graphs $\mathfrak{G}_l$  as the set of pairs $(G,m)$, where $G \in \mathfrak{G}$ and  
$$m= \{ E_0, E_1,..., E_l \}$$ 
with $E_0 \subset E_1 \subset ... \subset E_l \subset E(G)$ be an increasing finite sequence of subsets of $E(G)$.
On the set  $\mathfrak{G}_l$ we can extend straightforwardly the operator $\delta_e$ for $e \in D(G) \sqcup (E^{in}(G) \setminus E_l)$.  Let $\mathfrak{G}_*= \sqcup_l \mathfrak{G}_l$.

 



To the pair $(G,m) \in \mathfrak{G}_{l}$ we associate the moduli space
\begin{equation}    \label{multi-curve-interp0}
 \overline{\mathcal{M}}_{G,m} :=    \left( \prod_{c \in C(G)} \overline{\mathcal{M}}_{g_c,D_c, V_c, (H_v)_{v \in V_c} }(\beta_c)  \right)  / \text{Aut}(G,m) \times \Delta^l
 \end{equation} 
 where $\Delta^l$ is the standard simplex  of dimension $l$.

To each $e \in E^{in}(G) $,  the evaluation map on the punctures associated to $e$, defines a map $\text{ev}_e:  \overline{\mathcal{M}}_{G,m} \rightarrow L \times L$.  For $e \in E^{in}(G) \setminus E_l$, set the closed subspace of $\overline{\mathcal{M}}_{G,m}$:
\begin{equation} \label{bedge}
\delta_e \overline{\mathcal{M}}_{G,m}  := \left( (\text{ev}_e)^{-1}(Diag) \right) /  \text{Aut}(G,m,e).  
\end{equation}


For $e \in D(G)$ set
\begin{equation} \label{bcomponent}
\delta_e \overline{\mathcal{M}}_{G.m}  =  \left( (\text{ev}_e)^{-1}(L)  \right) /  \text{Aut}(G,,m,e).  
\end{equation}


Given an element $(G,m) \in \mathfrak{G}_l$ set
$$ \mathfrak{G}(G,m) = \{ (G',m' , E') | \thickspace (G',m') \in \mathfrak{G},  \thickspace E' \subset E(G') \setminus E_l', \thickspace G'/E' \cong G, \thickspace m' \subset m \}  .$$ 
In the condition $m' \subset m$ we use the isomorphism $G'/E' \cong G$ in order to identify the elements of $m'$ with  a set of edges of $G$. 

  

\subsubsection{Forgetful Compatibility} \label{forget-multi-moduli}



Let $(G,m) \in \mathfrak{G}_l$ and $e \in E^{ex}(G)$. We want to define a decorated graph $(G',m')= forget_e(G,m) \in \mathfrak{G}_l$ which is obtained removing the edge $e$.  The definition of $(G',m')$   is straightforward in the case that $G$ is stable after removing $e$ .  
If $G$ becomes unstable after removing $e$ we proceed as follows.

Let $v \in V(G)$ and $c \in Comp(G)$ such that $v \in V_c$ and $e \in H_v$. Since $c$ is unstable after removing $e$ we have $\beta_c=0$ and $g_c = 0$.
Let $G_e$ be the decorated graph defined by
$$\beta(G_e)= 0 , Comp(G_e)= \{  c \}, V(G_e)= V_c, D(G_e)= D_c,$$
$$ H(G_e)=H_c, E^{in}(G_e) = \{ e \in E^{in}(G)| e \subset H_c \} .$$
There are the following cases:
\begin{enumerate} \label{unstable-cases}
\item  $|V_c|=1, |D_c|= 0, |H_c| =3 ,  |E^{in}_c|=  0$; \label{disk}
\item  $|V_c|=1, |D_c|=0$,  $|H_c| =3 $,  $|E^{in}_c|=1$;  \label{disk-edge}
\item   $|V_c|=2, |D_c|=0$ , $|H_c| = 1 $;   \label{annelus} 
\item    $|V_c|=1, |D_c|=1$ ,$|H_c| = 1 $.   \label{annelus-deg}
\end{enumerate}

Denote by $Disk0$ the graph defined by  (\ref{disk}).
Let $Ann0$  be the set graphs given by (\ref{disk-edge}), (\ref{annelus}) and (\ref{annelus-deg}).

 We say that $e$ is not removable if $G_e$ is given by (\ref{disk-edge} )  and  $E^{in}_c \subset E_l$ . In all the other cases we say that $e$ is removable.


In the case (\ref{disk}), define $G'$ by removing the component $c$ and gluing the two elements of $H_v \setminus \{ e \}$. More precisely, let $H(G_e)= \{ h_1, h_2, e \}$.  If there exists $h_2' \in H(G)$ such that $h_1 \in E^{ex}(G)$, $\{ h_2, h_2' \} \in E^{in}(G)  $, declare $h_2' \in E^{ex}(G')$. 
  If there exists $h_1', h_2' \in H(G)$ such that $\{ h_1, h_1' \} \in E^{in}(G)$, $\{ h_2, h_2' \} \in E^{in}(G)  $
set $\{ h_1', h_2' \} \in E^{in}(G')$.
If $\{ h_1, h_1' \} \in E_i,  \{ h_2, h_2' \} \in E(G) \setminus E_{i-1} $  set  $\{ h_1', h_2' \} \in E_i$. 

In the cases $G_e \in Ann0$,
$G_e$ is a connected component of $G$ and we define $G'$ removing $G_e$ from $G$.

Assume that $e$ is removable. We say that the Kuranishi structure $ \overline{\mathcal{M}}_{G,m}$ is forgetful compatible with respect to $e$ if the following happen:
\begin{itemize}
    \item If $G$ is stable after removing $e $  we require that the Kuranishi structure of  
$ \overline{\mathcal{M}}_{G,m}$ is the pull-back of the Kuranishi structure of  $\overline{\mathcal{M}}_{G',m'} $
 as in subsection \ref{forgetful}.
 \item In the case $G$ is unstable after removing $e$  we require that:
\begin{equation} \label{forgetful-compatibility-multi-curves}
 \overline{\mathcal{M}}_{G,m}= \overline{\mathcal{M}}_{G',m'} \times  \overline{\mathcal{M}}_{G_e} .
\end{equation}
\end{itemize}

Assume that there are not external removable external edges. Let  $G^{\clubsuit}$ be the subgraph of $G$ which is the union of the connected components isomorphic to (\ref{disk-edge}). 
Write 
$ (G,m) = (G',m') \sqcup (G^{\clubsuit}, m^{\clubsuit} ) $.  $G'$ is a subgraph of $G$ without external edges.

For each integer $l \geq 0$ fix  a  decomposition of the standard $l$-simplex $\Delta^l= [0,1,...,l]$ 
\begin{equation} \label{fact-symplex}
     \Delta^l = (\bigsqcup_{0 \leq v \leq l} [0,1,...,v] \times [v,v+1 ...,l] )/ \sim 
\end{equation}
where the face $\partial_{v+1} [0,1,...,v, v+1] \times [v+1, ...,l]$ of  $[0,1,...,v, v+1] \times [v+1, ...,l]$ is identified with the face $ [0,1,..., v] \times \partial_{v} [v, ...,l]$ of  $[0,1,...,v] \times [v, ...,l]$. 
We assume that the decomposition induced on each boundary face of $\Delta^l$  is compatible with the decomposition of $\Delta^{l-1}$ .

Using (\ref{fact-symplex}) we obtain  an identification of moduli spaces 
\begin{equation} \label{moduli-factorization}
 \overline{\mathcal{M}}_{ G,m} = 
(\sum_{0 \leq r \leq l}  \overline{\mathcal{M}}_{G', m'_{[0,r]}} \times   \overline{\mathcal{M}}_{G^{\clubsuit}, m^{\clubsuit}_{[r,l]}} )/ \sim.
 \end{equation}
We require that $ \overline{\mathcal{M}}_{ G,m}$ is equipped with the Kuranishi structure  compatible with the right side of (\ref{moduli-factorization}).



\subsubsection{Kuranishi Structures}

We are interested in equipping the moduli spaces  $ \overline{\mathcal{M}}_{G,m}$  with a Kuranishi structure with the following properties:
\begin{enumerate}
\item $ \overline{\mathcal{M}}_{G,m}$ is forgetful compatible;
\item the evaluation map 
$$\text{ev} : \overline{\mathcal{M}}_{G,m} \rightarrow L^{H(G) \setminus H_l} \times X^{D(G)}$$ 
is weakly submersive;
\item   the corner faces of $\overline{\mathcal{M}}_{(G,m)}$ are in bijection with  the graphs $\mathfrak{G}(G,m)$
$$  \mathfrak{G}(G,m ) \leftrightarrow  \{    \text{   corner faces of  } \overline{\mathcal{M}}_{(G,m)}    \} .$$
The corner face $ \overline{\mathcal{M}}_{G,m }( G',m',E') $ corresponding to  $(G',m',E') \in \mathfrak{G}(G,m)$ comes with an identification of Kuranishi spaces: 
\begin{equation} \label{corner-faces-multi-interp0}
 \overline{\mathcal{M}}_{G, m }(G' ,m', E') \cong   \delta_{E'}  \overline{\mathcal{M}}_{G',m'}.
\end{equation}
\end{enumerate}

Note that in point $(3)$ it is used point $(2)$ to define the Kuranishi structure on $\delta_{E'}  \overline{\mathcal{M}}_{G',m'}$ as fiber product. 

If $E_k = E_{k+1}$ we require that 
$\overline{\mathcal{M}}_{G,m}  $ is the pull-back of $\overline{\mathcal{M}}_{G,\partial_k m}  $ by the map 
\begin{equation} \label{kuranishi-degenerate-marking}
    \overline{\mathcal{M}}_{G,m} \rightarrow \overline{\mathcal{M}}_{G,\partial_k m}  
\end{equation}
inducted by 
$\Delta^l \rightarrow \Delta^{l-1}, (t_k,t_{k+1}) \mapsto (t_k + t_{k+1})$.

We require  
\begin{equation} \label{kuranishi-cut}
    \overline{\mathcal{M}}_{G,m} = \overline{\mathcal{M}}_{cut_{E_0}G,m}.
\end{equation}

We require the following compatibility conditions for (\ref{corner-faces-multi-interp0}). Assume we have $G'' \prec G' \prec G$ with  $G''/E''=G'$,  $G'/E'=G$. In particular we have $G''/(E' \sqcup E'')= G$,  where $E'$ is identified with a subset of $E(G'')$.
There is an identification of Kuranishi spaces 
\begin{equation} \label{substrata-interp0}
\overline{\mathcal{M}}_{G,m} (G'',m'', E' \sqcup E'') \cong \delta_{E'}  \overline{\mathcal{M}}_{G',m'} \cap  \overline{\mathcal{M}}_{G',m'} (G'',m'', E'') \cong \delta_{E' \sqcup E''}  \overline{\mathcal{M}}_{G'',m''}.
\end{equation}
The spaces appearing on (\ref{substrata-interp0})  are closed subspace of $\overline{\mathcal{M}}_{G,m}$, $\overline{\mathcal{M}}_{G',m'}$ , $\overline{\mathcal{M}}_{G'',m''}$ respectively.


\begin{lemma} There exists a Kuranishi structure on the moduli space of multi-curves satisfying the above conditions.
\end{lemma}
\begin{proof}
The Lemma  is proved using a similar inductive argument of Lemma \ref{argument}.

Let $(G,m) \in \mathfrak{G}_l$, and assume that we have defined the Kuranishi structures for $(G'm')$ with $G' \prec G$ or $G=G'$ and $|m'| < l$.
 

If $ E^{ex}(G) \neq \emptyset $ ,  use forgetful compatibility to define  the Kuranishi structure on   $ \overline{\mathcal{M}}_{G,m}$ from the   Kuranishi structure of  
$ \overline{\mathcal{M}}_{forget_{e}(G,m)}$.


In the case there are not removable edges we apply the argument of Lemma \ref{argument}:  first use (\ref{corner-faces-multi-interp0}) to define a Kuranishi structure on a small neighborhood of the boundary of  $ \overline{\mathcal{M}}_{G,m}$ (observe that the inductive hypotheses  imply that the conditions are compatible on a neighborhood of the corner) and then extend the Kuranishi structure inside $ \overline{\mathcal{M}}_{G,m}$ using the argument around (\ref{obstruction-modification}).
\end{proof}

The following Lemma is an immediate consequence  of Lemma \ref{orientation-curves} and the  definition of $\mathfrak{o}_G$
\begin{lemma}
A spin structure on $L$ induces on  $ \overline{\mathcal{M}}_{G,m}$ an orientation with twisted coefficients $\mathfrak{o}_{H(G)}$.
\end{lemma}


Now set
 \begin{equation}    \label{multi-curve-interp}
\overline{\mathcal{M}}_{G,m}^K  :=   \overline{\mathcal{M}}_{G,m} \times_{ X^{D(G)} }   K^{D(G)} .
\end{equation} 
B $\overline{\mathcal{M}}_{G,m}^K $ is endowed of a Kuranishi Structure induced from the Kuranishi Structure of $\overline{\mathcal{M}}_{G,m} $  as fiber product (here we use property $(2)$ above ) .  From the properties of $\overline{\mathcal{M}}_{G,m} $, the Kuranishi spaces $\overline{\mathcal{M}}_{G,m}^K $ is endowed with the following properties:
\begin{enumerate}
\item $ \overline{\mathcal{M}}_{G,m}^K$ is forgetful compatible;
\item the evaluation map $\text{ev} : \overline{\mathcal{M}}_{G,m}^K \rightarrow L^{H(G) \setminus H_l} $ is weakly submersive;
\item   
to  each element $(G',m', E') \in \mathfrak{G}(G,m)$ corresponds a corner face $ \overline{\mathcal{M}}_{G,m }^K( G',m',E') $ of $\overline{\mathcal{M}}_{G,m}^K $  which comes with  an identification of Kuranishi spaces: 
\begin{equation} \label{corner-faces-multi-interp}
 \overline{\mathcal{M}}_{G, m }^K(G' ,m',E') =   \delta_{E'}  \overline{\mathcal{M}}_{G',m'}^K.
\end{equation}
\end{enumerate}
From property $(3)$ there is a map
$$  \mathfrak{G}(G,m ) \rightarrow  \{    \text{   corner faces of  }
 \overline{\mathcal{M}}_{(G, m \})} ^K   \} .$$
 which is injective but it is not surjective since the boundary faces of $\overline{\mathcal{M}}_{(G,m)}^K$ corresponding to the boundary of $K$ are missing  in  $\mathfrak{G}(G,m )$ .



The compatibility condition (\ref{substrata-interp0}) is extended straightforwardly: 
\begin{equation} \label{substrata-interp}
\overline{\mathcal{M}}_{G,m}^K (G'',m'', E' \sqcup E'') \cong \delta_{E'}  \overline{\mathcal{M}}_{G',m'}^K \cap  \overline{\mathcal{M}}_{G',m'}^K (G'',m'', E'') \cong \delta_{E' \sqcup E''}  \overline{\mathcal{M}}_{G'',m''}^K.
\end{equation}


\subsubsection{Parametric version} \label{parametric-section}

The moduli spaces we considered so far are associated to a fixed almost complex structure $J$.
For a fixed $J$, the Kuranishi structures we considered are not canonical, but they depend on varies choices we made during their construction. 
We  now consider the parametric version of the moduli spaces above, which define isotopies between  two different choice of  Kuranishi structures constructed as above. 

Let  $J_0$, $J_1$ be two almost complex structures compatible to the symplectic structure $\omega$.
Fix a Kuranishi structure on $\overline{\mathcal{M}}^0_{G,m} (J_0)$ and $\overline{\mathcal{M}}^1_{G,m} (J_1)$ constructed as in Subsection \ref{Moduli-interp-section}.
  
Let $J_{para}= \{ J_s \}_{0 \leq  s \leq 1}$ be one parameter family of compatible almost complex structures such that $J_s=J_0$ for $s \in [0, \epsilon]$ and $J_s=J_1$ for $s \in [1- \epsilon, 1]$.   
Let 
$$\hat{\overline{\mathcal{M}}}_{(G,m)}(J_{para}) := \bigsqcup_s \{ s \} \times \overline{\mathcal{M}}_{(G,m)}(J_s)$$
be  the parametric version of the moduli of multi-curves.




It is straightforward to extend the notion of forget-compatibility to $ \hat{\overline{\mathcal{M}}}_{G,m}$.

We can endow the moduli spaces  $ \hat{\overline{\mathcal{M}}}_{G,m}$ with a Kuranishi structure such that
\begin{itemize}
\item $ \overline{\mathcal{M}}_{G,m}$ is forgetful compatible  ;
\item the evaluation map $\text{ev} : \hat{\overline{\mathcal{M}}}_{G,m} \rightarrow L^{H(G) \setminus H_l} \times X^{D(G)} \times [0,1]$ is weakly submersive;
\item  to each element $(G',E',m') \in \mathfrak{G}(G,m)$ corresponds corner faces  $ \hat{\overline{\mathcal{M}}}_{G,m }( G',E',m',para)$  , $\hat{\overline{\mathcal{M}}}_{G,m }( G',E',m',0)$  , $\hat{\overline{\mathcal{M}}}_{G,m }( G',E',m',1) $  coming with  identification of Kuranishi spaces: 
\begin{align} \label{corner-faces-multi-interp-para}
\begin{split}
\hat{\overline{\mathcal{M}}}_{G,m} ((G', E' ,m', para)  \cong \delta_{E'} \hat{\overline{\mathcal{M}}}_{G',m'} ,  \\
\hat{\overline{\mathcal{M}}}_{G,m} ( G' ,E' , m' ,0)  \cong \delta_{E'} \overline{\mathcal{M}}_{G',m'}^0 , \\
\hat{\overline{\mathcal{M}}}_{G,m} (  G', E',m' ,1)  \cong \delta_{E'}  \overline{\mathcal{M}}_{G',m'}^1 .
\end{split}
\end{align}
\end{itemize}
The corner faces of   $\hat{\overline{\mathcal{M}}}_{G,m}$
are in bijection with  the set $\mathcal{G}(G,m)  \times \{0,para,1 \} $
$$  \mathfrak{G}(G,m ) \times \{0,para,1 \} \Longleftrightarrow  \{    \text{   corner faces of  } \hat{\overline{\mathcal{M}}}_{(G,m)}    \} .$$

We require that the identification (\ref{corner-faces-multi-interp-para}) satisfy compatibility conditions analogous to (\ref{substrata-interp}).

\subsubsection{Dependence on the four chain}

So far, in the definition of the moduli spaces   $\overline{\mathcal{M}}_{G,m}^K $ we have  fixed a  four chain $K$. We now consider the question of how our construction  depend    on the choice of the representative $K$ of $[K] \in H_4(X,L) $.

Let $K_0$ and $K_1$ be two four chains with $[K_0]=[K_1] \in H_4(X,L)$. Let $\tilde{K} \in C_5(X \times [0,1])$ such that
$$ \partial \tilde{K} = K_0 \times {0} - K_1 \times {1} + L \times [0,1]  .$$


We would like to use $\tilde{K}$ to define an isotopy of moduli space between the moduli spaces $\overline{\mathcal{M}}_{G,m}^{K_0} $ and $\overline{\mathcal{M}}_{G,m}^{K_1} $ as fiber product  $ \hat{\overline{\mathcal{M}}}_{G,m} \times_{X^{D(G)} \times [0,1]}  (\tilde{K}^{n} \cap X^n \times Diag_{[0,1]}) $. However we need to take in consideration that  $\tilde{K}^n$ is a  $4n+1$-chain on $(X \times [0,1])^n$ which in general is not transversal to $X^n \times Diag_{[0,1]}$ ( here we denote by $Diag_{[0,1]} \cong [0,1]$  the small diagonal of $[0,1]^n$). To get around this problem instead to  use  $\tilde{K}^n$ we shall use a chain  $\tilde{K}^{[n]} \in C_{4n +1}(X^n \times [0,1])$ with the following properties:
\begin{itemize}
\item  
$$ \tilde{K}^{[n]} =  \tilde{K}^{n,+} \cap  (X^n \times Diag_{[0,1]}) .$$
for some  chain $\tilde{K}^{n, +}  \in C_{5n}(X^n \times [0,1])$ close in the $C^0$-topology to $\tilde{K}^n$ and  transversal to  $X^n \times Diag_{[0,1]}$.

\item it is invariant by the action  of the group of permutations $S_n$ of the factors of $X^n$;
\item $   \partial \tilde{K}^{[n]} = ( K_0^n \times  \{0 \} - K_1^n \times  \{ 1 \} +   \tilde{K}^{n-1} \times L )/{(S_n)}. $
\end{itemize}
The existence  $\tilde{K}^{[n]}$ can be proved easily by induction on $n$.

Set
$$   \hat{\overline{\mathcal{M}}}_{G,m}^{\tilde{K}} =  \hat{\overline{\mathcal{M}}}_{G,m} \times_{X^{D(G)} \times [0,1]}  \tilde{K}^{D(G)}.$$
This space has analogous properties to $\hat{\overline{\mathcal{M}}}_{G,m}^K$ stated above.

The isotopy is not uniquely determined up to isotopy of isotopies.  The set of equivalence classes of isotopies is a  torsor on $H_5(X, \Z)$:
$$ equivalence \; classes \;  of \; isotopies  \Longleftrightarrow  \{  S \in C_5(X) |  \partial S = K_1 -K_0  \}/ \partial C_6(X) .$$







\subsection{Perturbation of the moduli of multi-curves}

We now want to define constrains on the perturbation of the moduli space of multi-curves such that the associated virtual fundamental classes  yield to a $MC$-cycle.

We start defining the notion of  forgetful compatible perturbation. 
This notion will  depends on the choice of a perturbation of  the moduli space of area zero annulus. 
In the following we use the  notation of subsection \ref{forget-multi-moduli}.

Fix a transversal perturbation 
\begin{equation} \label{perturbation-ann0}
(\mathfrak{s}_{G})_{G \in Ann0}
\end{equation}
 of $(\overline{\mathcal{M}}_{G}^K)_{G \in Ann0}$ compatible with the identification of Kuranishi spaces (\ref{corner-faces-multi-interp}). 
Note that the boundary of  (\ref{annelus}) are identified with $\text{ev}^{-1}(Diag_L)$ of (\ref{disk-edge}) and with the boundary face of (\ref{annelus-deg}).

Also set the perturbation of  the moduli space of area zero disks with $3$ marked points 
$\mathfrak{s}_{Disk_0}$
as its Kuranishi map, i.e. , we do not perturb its moduli space.

We say that the perturbation $\mathfrak{s}_{G,m}$ of  the Kuranishi space $\overline{\mathcal{M}}_{G,m}$  is forgetful compatible
if it is related to the perturbation 
 $\mathfrak{s}_{G',m'}$ of  $\overline{\mathcal{M}}_{G',m'}$ 
in the following way:
\begin{itemize}
    \item 
If $G$ is stable after removing $e$ we require that $\mathfrak{s}_{G,m}$ is the pull-back of  $\mathfrak{s}_{G',m'}$.
\item
In the case $G$ is unstable after removing $e$ and $e$ is removable we require that 
\begin{equation}  \label{forgetful-compatibility-multi-curves-pert}
\mathfrak{s}_{G,m} = \mathfrak{s}_{G',m'} \times \mathfrak{s}_{G_e}
\end{equation}
according the identification of Kuranishi spaces (\ref{forgetful-compatibility-multi-curves}).
\end{itemize}

If $(G,m)$ has not removable edge, we require that $\mathfrak{s}_{G,m}$ is compatible with (\ref{moduli-factorization}).

Observe that if $ \mathfrak{s}_{G',m'} $ is transversal to the zero section, the same holds for  $\mathfrak{s}_{G,m}$.

Let $ \mathfrak{G}(\beta,  \kappa   )$ be the set of $G \in  \mathfrak{G}(\beta)$ such that $|E^{ex}(G)| - \chi(G) = \kappa$.  Set $ \mathfrak{G}(\beta,  \leq \kappa   ) = \sqcup_{\kappa' \leq \kappa} \mathfrak{G}(\beta,  \kappa'   )$.
\begin{proposition} \label{transversality}
Fix $\beta \in H_2(X,L)$ and $\kappa \in \Z_{\geq 0}$. Also fix  perturbations of area zero annulus (\ref{perturbation-ann0}) .
There exists e collection of perturbations 
$( \mathfrak{s}_{(G,m)} )_{G \in \mathfrak{G}(\beta, \leq \kappa   ) }$ 
of the collection of Kuranishi spaces
$(\overline{\mathcal{M}}_{G,m}^K )_{G \in \mathfrak{G}(\beta, \leq \kappa   ) }$ such that 
\begin{enumerate}
\item they are transversal to the zero section and small in order for constructions of virtual class to work;
\item they are compatible with the identification of Kuranishi spaces (\ref{corner-faces-multi-interp});
\item 
  forgetful compatibility holds;
  \item they are compatible with (\ref{kuranishi-cut}) and (\ref{kuranishi-degenerate-marking}).
\end{enumerate}
\end{proposition}

\begin{proof}



Let $G \in  \mathfrak{G}(\beta,   \kappa   )$.
Assume that we have  defined $\mathfrak{s}_{(G',m')} $ for $G' \in  \mathfrak{G}(\beta, < \kappa   )$ or $G' \prec G$ or $G'=G$ and $|m'| < l$
such that $(1),(2),(3)$ hold. Assume  moreover 
\begin{itemize}
    \item (extra-condition) $\mathfrak{s}_{(G',m')} $  is transversal to the zero section when restricted to the closed spaces $\delta_{E'} \overline{\mathcal{M}}_{(G',m')}^K$ for each $E' \subset E^{in}(G') \setminus E_{l'}'$. 
\end{itemize} 

If $ E^{ex}(G) \neq \emptyset$ removable, use condition $(3)$ to define $\mathfrak{s}_{(G,m)}$ in terms of $\mathfrak{s}_{forget_e(G),m'}$.
It is immediate to check that  $\mathfrak{s}_{(G,m)}$  satisfies conditions $(1)$, $(2)$ and the extra condition above.

Assume now that there are not  $ E^{ex}(G) = \emptyset$.
Observe that condition $(2)$ 
defines  $\mathfrak{s}_{(G,m)}$ on the closed subset of 
$$ \bigsqcup_{(G',m') \prec (G,m)} \overline{\mathcal{M}}_{(G,m)}^K (G',m') \subset \partial \overline{\mathcal{M}}_{(G,m)}^K.$$ 
The fact that the $\mathfrak{s}_{(G',m')}$  are compatible on the overlap substrata follows from the inductive hypothesis and the compatibility condition (\ref{substrata-interp}) :
$$  \mathfrak{s}_{(G',m')}|_{ \overline{\mathcal{M}}_{(G',m')}^K (G'',m'',E' \sqcup E'') \cap \delta_{E'}  \mathcal{M}_{G',m'}^K} =  \mathfrak{s}_{G'',m''} |_{ \delta_{E' \sqcup E''}  \mathcal{M}_{G'',m''}^K } .$$
The extra condition above assures that $\mathfrak{s}_{(G,m)}$ is transversal on this subset.


Therefore we could try to use Lemma \ref{perturbation-lemma-FO}
in order to define $\mathfrak{s}$ on $\overline{\mathcal{M}}_{G,m}^K$ such that compatibility 
condition $(2)$ holds. However we need to require that the perturbation on $\overline{\mathcal{M}}_{G,m}^K$ has to be small in order for constructions of virtual class also work. This is probably incompatible on (\ref{corner-faces-multi-interp}) with the perturbations we have defined in the preview steps.  To solve this problem we need to choose at every step the perturbations $\mathfrak{s}_{G,m}$  sufficiently close to $s_{G,m}$ in $C^0$-topology so that not only it makes the construction of virtual class for $\overline{\mathcal{M}}_{G,m}^K$ work, but also 
the prescribed values for $\mathfrak{s}$ at later inductive steps in the proof should be sufficiently close to $s$ that the later constructions of virtual class also work. (An analogous issue arise  in $7.2.56$ of \cite{FO3}.) Moreover we can assume generically that the extra condition holds also.




\end{proof}

\subsubsection{Main theorem}

For each $(G,m) \in \mathfrak{G}_{l}(\beta, \chi)$, the evaluation map on the punctures defines strongly continuous map 
\begin{equation} \label{higher-ev}
\text{ev}_{G,m} : \overline{\mathcal{M}}_{G,m}^K \rightarrow L^{H(G)}. 
\end{equation}

Using the perturbation of the Kuranishi spaces defined in Proposition \ref{transversality} set 
\begin{equation} \label{multi-curve-cycle-not-ab}
 Z_{(G,m)}^{not-ab} =  (\text{ev}_{G,{m}})_* (\mathfrak{s}_{G,{m}}^{-1}(0)).
\end{equation}

Recall that in the preview section in order to define $MCH$ we have used a different kind of set of decorated graphs $\mathcal{G}$.  We can go from $ \mathfrak{G}$ to  $\mathcal{G}$ using what we call abelianization map:
$$ab: \mathfrak{G} \rightarrow \mathcal{G}$$ 
$$ (Comp, (V_c, D_c, \beta_c ,g_c )_c, (H_v)_v, E) \mapsto (V^{ab} , (H_{v'}^{ab}, \chi_{v'}^{ab}, \beta_{v'}^{ab})_{v'}, E^{ab} )$$
where
$$ V^{ab} = Comp, H_{c}^{ab} = \sqcup_{v \in V_c} H_v, \beta_{c}^{ab}= \beta_c , \chi_{c}^{ab} = 2 - 2 g_c - |D_c| - |V_c|, E^{ab} = E.$$
Note that in the definition of $H_{v'}^{ab}$ it is forgot the cyclic order of $H_v$.


Set
\begin{equation} \label{multi-curve-cycle}
Z_{(G^{ab},m)} = \sum_{ab(G,m)=(G^{ab},m)}  Z_{(G,m)}^{not-ab}
\end{equation} 
for each $(G^{ab},m) \in \mathcal{G}(\beta, \chi)$.

From point $(2)$ of Proposition \ref{transversality} it follows directly that $\hat{\partial} Z =0$ .
The fact that $Z$ is forgetful compatible  follows essentially from the assumption $(3)$  of Proposition \ref{transversality} . 
However this claim does not hold in a strict sense since the  unstable components after removing the external edges considered in section (\ref{forget-multi-moduli}) do not lead to a trivial chain.
Hence there can the graphs $(G,m)$ considered in section \ref{forget-MC} with $V(G'') \setminus V(G') \neq V^{ann}$ but $Z_{G,m} \neq 0$. 
 However these chains can be contracted to zero in an essentially unique way in (\ref{multi-curve-cycle}).  We can use this contraction to define an isotopy between (\ref{multi-curve-cycle}) and a $MC$-cycle for which forgetful compatibility holds in the sense of section \ref{forget-MC} . We shall refer to (\ref{multi-curve-cycle}) as the $MC$-cycle where this isotopy is understood.   



The $MC$-cycle (\ref{multi-curve-cycle}) depends on the varies choice we made in the construction we made so far, such as  the  compatible almost complex structure $J$, the Kuranishi structure and its perturbations.  
However we have the following main result:
\begin{theorem}  \label{well-defined}
Equation  (\ref{multi-curve-cycle})  defines a multi-curve-cycle.  Different choices in the construction lead to isotopic $MC$-cycles, with isotopy determined up to isotopy. 
\end{theorem}
\begin{proof}





Let $Z_1$ and $Z_2$ be two cycles corresponding to two different choices of a perturbation, etc. 
We construct the isotopy between  $Z_1$ and $Z_2$  using the parametric space of subsection \ref{parametric-section}. 

It is straightforward to extend forgetful compatibility to perturbations $ \mathfrak{s}^{para}_{(G,m)} $ of $\hat{\overline{\mathcal{M}}}_{G,m}$.

Let $\mathfrak{s}^0_{(G,m)}$ and $\mathfrak{s}^1_{(G,m)}$ be  perturbations of $\overline{\mathcal{M}}_{(G,m)}(J_0)$ and $\overline{\mathcal{M}}_{(G,m)}(J_1)$ satisfying the conditions of Proposition \ref{transversality}.

Using an  inductive argument analogous to the one of Proposition \ref{transversality}, we  can prove the existence of perturbations $ \mathfrak{s}^{para}_{(G,m)} $ of $\hat{\overline{\mathcal{M}}}_{G,m}$ such that 
\begin{enumerate}
\item are transversal to the zero section and small in order for constructions of virtual class to work;
\item are compatible with the identification of Kuranishi spaces (\ref{corner-faces-multi-interp-para});
\item 
 the forgetful compatibility   holds.
 \item
  $\mathfrak{s}_{(G',m')}^{para} $  is transversal to the zero section when restricted to the closed space $\delta_{E'} \overline{\mathcal{M}}_{(G',m')}^{para}$ for each $E' \subset E^{in}(G')$.  
 \end{enumerate}
 Observe that condition $(1)$ set extra conditions on $\mathfrak{s}_{(G,m)}(J_0)$ and $\mathfrak{s}_{(G,m)}(J_1)$. 
 



Now, for each $(G,m) \in \mathfrak{G}_{l}(\beta, \chi)$, the evaluation map on the punctures yields to a strongly continuous map 
\begin{equation} \label{higher-ev-para}
\text{ev}_{G,m}^{para} : \overline{\mathcal{M}}_{G,m}^{para} \rightarrow L^{H(G)} \times [0,1]. 
\end{equation}
which is the parametric version of the evaluation map (\ref{higher-ev}).

Set
$$ \tilde{Z}_{G,m}(J_{para}) =  (\text{ev}^{para}_{G,{m}})_* ((\mathfrak{s}_{G,{m}}^{para})^{-1}(0)) .$$
From properties $1,2,3$ above it follows that $(\tilde{Z}_{G,m}(J_{para}))_{G,m}$ defines  the required isotopy.   
\end{proof}

\end{document}